\documentclass[12pt]{amsart}%
\usepackage{amsmath}
\usepackage{amsfonts}
\usepackage{amssymb}
\usepackage{amstext}
\usepackage{graphicx}%
\providecommand{\U}[1]{\protect\rule{.1in}{.1in}}
\providecommand{\U}[1]{\protect \rule{.1in}{.1in}}
\providecommand{\U}[1]{\protect \rule{.1in}{.1in}}
\newtheorem{theorem}{Theorem}[section]
\newtheorem{lemma}{Lemma}[section]
\newtheorem{proposition}{Proposition}[section]
\newtheorem{corollary}{Corollary}[section]

\newtheorem{definition}{Definition}[section]

\numberwithin{equation}{section}
\newtheorem {conjecture}{Conjecture}

\theoremstyle{remark}
\newtheorem{remark}{Remark}[section]
\numberwithin{equation}{section}
\begin{document}
\title[Sasakian Structure of Manifolds with Nonnegative Transverse Bisectional Curvature]{On the Sasakian Structure of Manifolds with Nonnegative Transverse Bisectional Curvature}
\author{$^{\ast}$Shu-Cheng Chang}
\address{Shanghai Institute of Mathematics and Interdisciplinary Sciences, Shanghai
(SIMIS), 200433, China}
\email{scchang@math.ntu.edu.tw }
\author{$^{\dag}$Yingbo Han}
\address{{School of Mathematics and Statistics, Xinyang Normal University}\\
Xinyang,464000, Henan, China}
\email{{yingbohan@163.com}}
\author{$^{\ast\ast}$Chien Lin}
\address{Department of Mathematics, National Taiwan Normal University, Taipei, Taiwan}
\email{clin@math.ntnu.edu.tw}
\author{$^{\ddag}$Chin-Tung Wu}
\address{$^{^{\ddag}}$Department of Applied Mathematics, National Pingtung University,
Pingtung 90003, Taiwan}
\email{ctwu@mail.nptu.edu.tw }
\thanks{$^{\ast}$Shu-Cheng Chang is partially supported in part by Startup Foundation
for Advanced Talents of the Shanghai Institute for Mathematics and
Interdisciplinary Sciences (No.2302-SRFP-2024-0049). $^{\dag}$Yingbo Han is
partially supported by an NSFC grant No. 12571056 and NSF of Henan Province
No. 252300421497. $^{\ast\ast}$ Chien Lin is partially supported by NSTC
114-2115-M-003-002-MY2. $^{\ddag}$ Chin-Tung Wu is partially supported by NSTC
grant 114-2115-M-153-001, Taiwan. }

\begin{abstract}
In this paper, we concern with the Sasaki analogue of Yau uniformization conjecture
in a complete noncompact Sasakian manifold with nonnegative transverse
bisectional curvature. As a consequence, we confirm that any $5$-dimensional
complete noncompact Sasakian manifold with positive transverse bisectional
curvature and the maximal volume growth must be CR-biholomorphic to the
standard Heisenberg group $\mathbb{H}_{2}$ which can be stated as the standard
contact Euclidean $5$-space $\mathbb{R}^{5}$.

\end{abstract}
\maketitle
\tableofcontents

\section{Introduction}

A Sasakian manifold is an odd dimensional counterpart of K\"{a}hler geometry
which is also a strictly pseudoconvex CR $(2n+1)$-manifold of vanishing
pseudohermitian torsion. It is known He-Sun (\cite{hs}) that a simply
connected closed Sasakian manifold with positive transverse bisectional
curvature is CR equivalent to the round CR sphere $\mathbb{S}^{2n+1}$. Then it
is very natural to concern with an CR analogue of Yau uniformization
conjectures in a complete noncompact Sasakian $(2n+1)$-manifold with positive
transverse bisectional curvature.

We recall that the famous Yau's uniformization conjecture on a complete
noncompact K\"{a}hler manifold states that

\begin{conjecture}
\label{conj} If $M$ is a complete noncompact $m$-dimensional K\"{a}hler
manifold with positive holomorphic bisectional curvature, then $M$ is
biholomorphic to the standard $m$-dimensional complex Euclidean space
$\mathbb{C}^{m}$.
\end{conjecture}

There are many illustrations of this conjecture, among others, the
breakthrough progress relating to this conjecture could be attributed to
Siu-Yau (\cite{sy}) and Mok (\cite{m}). More precisely, a complete noncompact
$m$-dimensional K\"{a}hler manifold $M$ with nonnegative holomorphic
bisectional curvature is isometrically biholomorphic to $\mathbb{C}^{m}$ under
the assumptions of the maximum volume growth condition and the scalar
curvature decays in certain rate. Since then there are several further works
aiming to prove the optimal result, such as in Chen-Zhu (\cite{cz}), Ni
(\cite{n}) and Chau-Tam (\cite{ct1}). They proved that a complete noncompact
K\"{a}hler manifold of bounded nonnegative holomorphic bisectional curvature
and maximal volume growth is biholomorphic to $\mathbb{C}^{m}$. Moreover, Liu
(\cite{l2}) and Lee-Tam (\cite{lt2}) confirmed Yau's uniformization conjecture
when $M$ has the maximal volume growth property via the elliptic and parabolic
methods, respectively.

For the Sasaki setting, one can regard the standard Heisenberg group
$\mathbb{H}_{n}=\mathbb{C}^{n}\times\mathbb{R}$ of dimension $2n+1$ as the
standard contact Euclidean space $(\mathbb{R}^{2n+1},\eta_{\mathrm{can}}%
,\Phi,\xi,g_{\mathrm{can}})$ with coordinates $(x_{1},...,x_{n},y_{1}%
,...,y_{n},s)$ and taking $\eta_{\mathrm{can}}$ as the canonical contact
structure and $g$ as the canonical Sasaki metric as in (\ref{2026A}) and
(\ref{2026B}), respectively. Then as in our previous papers \cite{chl1} and
\cite{chll}, one can propose the following conjecture.

\begin{conjecture}
\label{conj5} If $M$ is a complete noncompact Sasakian $(2n+1)$-manifold of
positive transverse holomorphic bisectional curvature. Then $M$ is CR
biholomorphic to the standard Heisenberg group%
\[
\mathbb{H}_{n}=\mathbb{C}^{n}\times\mathbb{R}.
\]

\end{conjecture}

In view of all previous approaches toward Yau uniformization Conjectures, it
is important to know if there exists a nonconstant holomorphic functions of
polynomial growth in a complete noncompact K\"{a}hler manifold of nonnegative
holomorphic bisectional curvature and positive at some point. In fact, for the
family $\mathcal{O}_{d}(M)$ of all holomorphic functions of polynomial growth
of degree at most $d$. $\mathcal{O}_{p}\left(  M\right)  \neq\mathbb{C}$ if
and only if $M$ has the maximal volume growth property. Therefore, for the
family $\mathcal{O}_{d}^{CR}(M)$ of all basic CR-holomorphic functions $f$ of
polynomial growth of degree at most $d$, it is important to know when \
\[
\mathcal{O}_{p}^{CR}\left(  M\right)  \neq\mathbb{C}%
\]
if one wishes to work on Conjecture \ref{conj5}.

From this view point, we first affirmed the part of CR Yau uniformization conjecture on Sasakian manifolds.

\begin{proposition}
\label{P1A} (\cite{chll}) There exists a nonconstant CR holomorphic function
of polynomial growth in a complete noncompact Sasakian $(2n+1)$-manifold of
nonnegative transverse bisectional curvature with the maximal volume growth
property%
\begin{equation}%
\begin{array}
[c]{c}%
\lim_{r\rightarrow+\infty}\frac{\mathrm{Vol}\left(  B_{\xi}\left(  p,r\right)
\right)  }{r^{2n}}\geq\alpha,
\end{array}
\label{2020AAA}%
\end{equation}
for a fixed point $p$ and some positive constant $\alpha$. Here $B_{\xi
}\left(  p,r\right)  $ is the horizontal ball in a Sasakian $(2n+1)$-manifold.
\end{proposition}

Our method is to construct CR-holomorphic functions with controlled growth in
a sequence of exhaustion domains on $M$ \ via the Cheeger-Colding theory, heat
flow technique and Hormander $L^{2}$-estimate in Sasakian manifolds. Then by
applying the CR analogue of tangent cone at infinity and three circle theorem
to ensure that we can take the subsequence to obtain a nonconstant basic CR
holomorphic function of polynomial growth.

Furthermore, as a consequence of Proposition \ref{P1A} plus the method as in
\cite{l1}, there exist finitely many polynomial growth basic CR-holomorphic
basic functions $f_{1},\cdots,f_{k}$ so that $(f_{1},\cdots,f_{k},t)$ is a
proper basic CR-holomorphic map from $M$ to $\mathbb{C}^{k}\times\mathbb{R}$.
Let $n_{k}=\dim(\mathcal{O}_{k}^{CR}(M))$ for any given $k\in\mathbb{N}$ with
a basis $\{g_{n_{1}},\cdots,g_{n_{k}}\}.$ Define a basic CR-holomorphic map
from $M$ to $\mathbb{C}^{n_{k}}\times\mathbb{R}$ as
\[
F_{k}(x)=(g_{n_{1}}(x),\cdots,g_{n_{k}}(x),t)
\]
with $F_{k}(M)=\Sigma_{k}\times\mathbb{R}.$ Here $\Sigma_{k}$ is the affine
algebraic variety defined the integral ring generated by $g_{n_{1}}%
(x),\cdots,g_{n_{k}}(x)$ in $\mathbb{C}^{n_{k}}$. Hence

\begin{corollary}
Any complete noncompact Sasakian $(2n+1)$-manifold of nonnegative transverse
bisectional curvature with the CR maximal volume growth property is CR
biholomorphic to $\Sigma_{k}\times\mathbb{R}$ \ for an affine algebraic
variety $\Sigma_{k}$ in $\mathbb{C}^{n_{k}}$.
\end{corollary}

\begin{remark}
Let $M^{2n+1}$ be a Sasakian manifold as in section $2$ with a Sasakian
structure $(M^{2n+1},\eta,\xi,\Phi,g)$. Then $g|_{D}:=g^{T}$ is transverse
K\"{a}hler metric on the Reeb foliation space and
\[
g=g^{T}+\eta\otimes\eta.
\]
One can regard $(M^{2n+1},\eta,\xi,g)$ as a CR manifold of vanishing
pseudohermitian torsion with the Levi metric $d\eta$ and its associated
Webster adapted metric
\[
g=\frac{1}{2}d\eta+\eta\otimes\eta.
\]
The asymptotic maximal volume ratio can be read as
\[%
\begin{array}
[c]{c}%
\lim_{r\rightarrow\infty}\frac{\mathrm{Vol}_{C.C.}(B_{p}(r))}{r^{2n+2}}%
=\alpha_{0}>0
\end{array}
\]
and
\[%
\begin{array}
[c]{c}%
\lim_{r\rightarrow\infty}\frac{\mathrm{Vol}_{\xi}(B_{p}(r))}{r^{2n}}%
=\alpha_{0}>0,
\end{array}
\]
respectively (\cite{chll}).
\end{remark}

In this paper, following methods as in Chang-Han-Lin-Li (\cite{chll}),
Chen-Zhu (\cite{cz}) and Lee-Tam (\cite{lt2}), we study the Sasakian structure
of complete noncompact manifolds with nonnegative transverse Bisectional
curvature via the Sasaki-Ricci flow (\ref{2020C}). Indeed, by applying
proposition \ref{P1A}, we will construct CR biholomorphisms from a sequence of
open sets which exhaust $M$ onto a fixed ball in $\mathbb{C}^{n}%
\times\mathbb{R}$.

\begin{theorem}
\label{T1} Let $(M,\xi,\eta,\Phi,g,\omega)$ be a complete noncompact Sasakian
$(2n+1)$-manifold of nonnegative transverse bisectional curvature. Suppose
that the maximal volume growth holds, then

\begin{enumerate}
\item $M$ is CR-biholomorphic to $\Omega\times\mathbb{R},$ where $\Omega$ is a
pseuodoconvex domain in the $\mathbb{C}^{n}.$

\item In addition for $n=2$, if the transverse bisectional curvature is
positive, then $M$ must be CR-biholomorphic to the standard Heisenberg group
of dimension five
\[%
\begin{array}
[c]{c}%
\mathbb{H}_{2}=\mathbb{C}^{2}\times\mathbb{R}.
\end{array}
\]

\end{enumerate}
\end{theorem}

Theorem \ref{T1} will be a consequence of Theorem \ref{T2A} and Theorem
\ref{T2B}.\ Moreover, by a theorem of Ramanujam \cite{r} which states that any
non-singular complex algebraic surface of contractible and simply connected at
infinity must be isomorphic to the affine two-space as algebraic variety.
Then, as a consquence of Theorem \ref{T1} and Theorem \ref{T2A}, one can
confirm the CR Yau uniformization conjecture with the maximal volume growth
condition for $n=2$.

By using the standard method as in the paper of Shi \cite{s1}, we have the
short-tome solution of the Sasaki-Ricci flow in a complete noncompact Sasakian
$(2n+1)$-manifold with bounded transverse holomorphic bisectional curvature.
However, by applying the so-called Sasaki-Chern-Ricci flow (\ref{1}) which is
the Sasaki analogue of Chern-Ricci flow as in recent papers of Lee-Tam
(\cite{lt2}), Tosatti-Weinkove (\cite{tw}) and Lott-Zhang (\cite{lz}), we have
the Sasaki analogue of the short-time solution for Sasaki-Ricci flow\ by removing the boundedness of%
transverse holomorphic bisectional curvature.

\begin{proposition}
\label{P1} Let $(M,\xi,\eta,g_{0},\Phi,\omega_{0})$ be a complete noncompact
Sasakian $(2n+1)$-manifold with transverse nonnegative holomorphic bisectional
curvature and $V_{\xi}(x,1)\geq v_{0}$ for some $v_{0}>0$ and all $x\in M$. Then

\begin{enumerate}
\item there exist $T_{0}(n,v_{0})>0,C_{0}(n,v_{0})>0$ such that a family of
Sasakian structures $(M,\xi(t),\eta(t),g(t),\Phi(t),\omega(t))$ satisfying the
Sasaki-Ricci flow
\begin{equation}%
\begin{array}
[c]{lll}%
\frac{d}{dt}g_{i\overline{j}}^{T}(x,t) & = & -R_{i\overline{j}}^{T}(x,t),\\
g_{i\overline{j}}^{T}(x,0) & = & g_{i\overline{j}}^{T}(x).
\end{array}
\label{2020C}%
\end{equation}
with
\begin{equation}%
\begin{array}
[c]{c}%
||Rm^{T}||^{2}(x,t)\leq\frac{C_{0}}{t},\text{ \textrm{on}\ }M\times
\lbrack0,T_{0}).
\end{array}
\label{AAA}%
\end{equation}

\item $g(t)$ has transverse nonnegative holomorphic bisectional curvature and
the maximal volume growth condition is preserving.

\item
\[%
\begin{array}
[c]{c}%
V_{\xi}(x,1)\geq\frac{1}{2}v_{0}%
\end{array}
\]
on $M\times\lbrack0,T_{0}).$
\end{enumerate}
\end{proposition}

For a completeness, we will give a proof as in the appendix.

First we will have the
long-time solution of (\ref{2020C}):

\begin{theorem}
\label{T1A}Let $(M,\xi,\eta,g,\Phi,\omega)$ be a complete noncompact Sasakian
$(2n+1)$-manifold with nonnegative transverse bisectional curvature. Suppose
that there exists a constant $C_{2}>0$ such that
\[%
\begin{array}
[c]{c}%
\frac{1}{\mathrm{Vol}(B_{\xi}(x,r))}\int_{B_{\xi}(x,r)}R^{T}\omega^{n}%
\wedge\eta\leq\frac{C_{2}}{(1+r)^{2}},\text{\ \textrm{for all} }r,\text{ }x\in
M.
\end{array}
\]
Then the Sasaki-Ricci flow has a long time solution $g(x,t)$ on $M\times
\lbrack0,\infty).$ Moreover, the following are true:

\begin{enumerate}
\item For any $t\geq0$, $g_{i\overline{j}}^{T}(x,t)$ is transverse K\"{a}hler
with nonnegative transverse bisectional curvature.

\item For any integer $m\geq0$, there is a constant $C_{1}$ such that%
\begin{equation}%
\begin{array}
[c]{c}%
||\nabla^{m}Rm^{T}||^{2}(x,t)\leq\frac{C_{1}}{t^{1+m}},\text{\ \textrm{for
all} }t>0,\text{ }x\in M.
\end{array}
\label{BBB}%
\end{equation}

\item If in addition $(M,\eta,\Phi,g_{i\overline{j}}^{T}(x))$ has maximal
volume growth, then there exists a positive constant $C_{2}$ depending only on
the initial metric such that the injectivity radius of $g_{i\overline{j}%
}(x,t)$ is bounded below by $C_{2}t^{\frac{1}{2}}$ for all $t\geq1$.
\end{enumerate}
\end{theorem}

\begin{theorem}
\label{T2A}Let $(M,\xi,\eta,g,\Phi,\omega)$ be a complete noncompact Sasakian
$(2n+1)$-manifold with nonnegative transverse bisectional curvature. Suppose
that for some positive constants $C_{1},C_{2}$, we have the maximal volume
growth%
\[%
\begin{array}
[c]{c}%
\mathrm{Vol}(B_{\xi}(p,r))\geq C_{1}r^{2n},\text{ \ \textrm{for all} }r\text{
\textrm{and some }}p\in M
\end{array}
\]
and
\[%
\begin{array}
[c]{c}%
\frac{1}{\mathrm{Vol}(B_{\xi}(x,r))}\int_{B_{\xi}(x,r)}R^{T}\omega^{n}%
\wedge\eta\leq\frac{C_{2}}{(1+r)^{2}},\text{\textrm{ for all} }r,x\in M.
\end{array}
\]
Then

\begin{enumerate}
\item $M$ is CR-biholomorphic to $\Omega\times\mathbb{R},$ where $\Omega$ is a
pseuodoconvex domain in the $\mathbb{C}^{n}$.

\item $M$ is homeomorphic to $\mathbb{R}^{4}\times\mathbb{R}$ for $n=2$ and
$M$ is diffeomorphic to $\mathbb{R}^{2n}\times\mathbb{R}$ for $n>2.$
\end{enumerate}
\end{theorem}

Next we show that

\begin{theorem}
\label{T2B}Let $(M,J,\theta)$ be a complete noncompact Sasakian $(2n+1)$%
-manifold with nonnegative transverse bisectional curvature and the maximal
volume growth. Then we have
\[%
\begin{array}
[c]{c}%
\frac{1}{B_{\xi}(x,r)}\int_{B_{\xi}(x,r)}R^{T}\omega^{n}\wedge\eta\leq\frac
{C}{1+r^{2}}%
\end{array}
\]
for all $r$, $x\in M$ and a positive constant $C$.
\end{theorem}

Then Theorem \ref{T1} will be a consequence of Theorem \ref{T2A} and Theorem
\ref{T2B}.

The paper is organized as follows. In section $2$, we review some preliminary
notions in a Sasakian manifold. In section $3,$ we show the Li-Yau-Hamilton
estimates and the Harnack inequalities for the Sasaki-Ricci flow (%
\ref{2020C}
). In section $4,$ we get the
long-time solution existence of the Sasaki-Ricci flow
(\ref{2020C})
. Then we prove the Yau's uniformization conjecture for a complete noncompact
Sasakian $(2n+1)$-manifold with nonnegative transverse bisectional curvature
in section $5$. In the appendix, we give a sketch proof of the Proposition
\ref{P1}.

\section{Preliminary}

We will address the preliminary notions on the foliation normal coordinate and
Type II deformation in a Sasakian manifold. We refer to \cite{bg}, \cite{fow},
and references therein for some details.

Let $(M,g)$ be a Riemannian $(2n+1)$-manifold. $(M,g)$ is called Sasaki
if\ the metric cone
\[
(C(M),\overline{g},J):=(\mathbb{R}^{+}\times M\mathbf{,}dr^{2}+r^{2}g)
\]
such that $(C(M),\overline{g},J,\overline{\omega})$ is K\"{a}hler with
$\overline{\omega}=\frac{1}{2}\sqrt{-1}\partial\overline{\partial}r^{2}.$ The
function $\frac{1}{2}r^{2}$ is hence a global K\"{a}hler potential for the
cone metric. As $\{r=1\}=\{1\}\times M\subset C(M)$, we may define the Reeb
vector field $\xi$ on $M$
\[%
\begin{array}
[c]{c}%
\xi=J(\frac{\partial}{\partial r}).
\end{array}
\]
Also the contact $1$-form $\eta$ on $TM$ satisfies
\[
\eta(\cdot)=g(\xi,\cdot).
\]
Then $\xi$ is the Killing vector field with unit length such that $\eta
(\xi)=1\ $and$\ d\eta(\xi,X)=0.$ In fact, the tensor field $\Phi$ of type
$(1,1)$, defined by $\Phi(Y)=\nabla_{Y}\xi$, satisfies the condition%
\[
(\nabla_{X}\Phi)(Y)=g(\xi,Y)X-g(X,Y)\xi
\]
for any pair of vector fields $X$ and $Y$ on $M$. Then such a triple
$(\eta,\xi,\Phi)$ is called a Sasakian structure on a Sasakian manifold
$(M,g).$

Let $(M,\eta,\xi,\Phi,g)$ be a compact Sasakian $(2n+1)$-manifold and $F_{\xi
}$ be the Reeb foliation generated by $\xi$ with $g(\xi,\xi)=1$ and the
integral curves of $\xi$ are geodesics. For any point $p\in M$, we can
construct the local coordinates in a neighborhood of $p$ which are
simultaneously foliated and Riemann normal coordinates (\cite{gkn}). That is,
we can find Riemann normal coordinates $\{x,z^{1},z^{2},\cdots,z^{n}\}$ on a
neighborhood $U$ of $p$, such that $\frac{\partial}{\partial x}=\xi$ on $U$.
Let $(M,J,\theta,\xi,g_{\theta})$ be a complete Sasakian $(2n+1)$-manifold and
$F_{\xi}$ be the Reeb foliation generated by $\xi.$ Then $\xi-\sqrt{-1}%
J\xi=\xi+\sqrt{-1}r\frac{\partial}{\partial r}$ is a holomorphic vector field
on the cone $C(M)$ and there is an action of the holomorphic flow generated by
$\xi-\sqrt{-1}J\xi$ on the cone. The local orbits of this action defines a
transverse holomorphic structure on the Reeb foliation $F_{\xi}$. More
precisely, for an open covering $\{U_{\alpha}\}_{\alpha\in A}$ of the Sasakian
manifold with a submersion $\pi_{\alpha}:U_{\alpha}\rightarrow V_{\alpha
}\subset%
\mathbb{C}
^{n}$, there is a canonical isomorphism
\[
d\pi_{\alpha}:D_{p}\rightarrow T_{\pi_{\alpha}(p)}V_{\alpha}%
\]
for any $p\in U_{\alpha},$ where the contact subbundle $D=\ker\theta\subset
TM.$ $d\theta$ is the the Levi form on $D.$ It is a global two form on $M$.
Since $\xi$ generates isometries, the restriction of the Sasakian metric $g$
to $D$ gives a well-defined Hermitian metric $g_{\alpha}^{T}$ on $V_{\alpha}$
\[
g|_{D}\leftrightarrow g_{\alpha}^{T}\text{ \textrm{on} }V_{\alpha}.
\]
The transverse Hermitian metric on the Reeb foliation with transverse
holomorphic structure is denoted by $g_{i\overline{j}}^{T}$ and its associated
$\omega$ which is still called the transverse K\"{a}hler $(1,1)$-form. More
precisely, the foliation local coordinate $\{x,z^{1},z^{2},\cdots,z^{n}\}$ on
$U_{\alpha}\ $and $(D\otimes\mathbb{C})^{\left(  1,0\right)  }$ is spanned by
the form
\[%
\begin{array}
[c]{c}%
Z_{\alpha}=\frac{\partial}{\partial z^{\alpha}}-\eta(\frac{\partial}{\partial
z^{\alpha}})\frac{\partial}{\partial x},\ \mathrm{for}\ \alpha=1,2,\cdots,n.
\end{array}
\]
with
\[
\lbrack Z_{i},Z_{j}]=[\xi,Z_{j}]=0
\]
and its dual frame%
\[
\{\eta,dz^{j},d\overline{z}^{j}|\ j=1,2,\cdots,n\}
\]

The Levi-form $d\eta$ on $D$ and the corresponding transverse Hermitian metric
$g^{T}$ are to be
\[
g=g^{T}+\eta\otimes\eta
\]
and
\[%
\begin{array}
[c]{c}%
g_{i\overline{j}}^{T}=g^{T}(\frac{\partial}{\partial z^{i}},\frac{\partial
}{\partial\overline{z}^{j}})=d\eta(\frac{\partial}{\partial z^{i}},\Phi
\frac{\partial}{\partial\overline{z}^{j}}).
\end{array}
\]
In terms of the foliation normal coordinate, we have%
\[%
\begin{array}
[c]{c}%
g^{T}=g_{i\overline{j}}^{T}dz^{i}d\overline{z}^{j}\text{ \textrm{and} }%
\omega:=d\eta=\sqrt{-1}(g_{i\overline{j}}^{T})dz^{i}\wedge d\overline{z}^{j}.
\end{array}
\]
However, since $i(T)d\eta=0,$
\[%
\begin{array}
[c]{c}%
d\eta(Z_{\alpha},\overline{Z_{\beta}})=d\eta(\frac{\partial}{\partial
z^{\alpha}},\frac{\partial}{\overline{\partial}z^{\beta}}).
\end{array}
\]
This Hermitian metric in fact is K\"{a}hler. We often refer to $\omega:=d\eta$
as the K\"{a}hler form of the transverse K\"{a}hler metric $g^{T}$ in the leaf
space $D$.

In fact, for $\widetilde{X},\widetilde{Y},\widetilde{W},\widetilde{Z}\in
\Gamma(TD)$ and the $d\pi_{\alpha}$-corresponding $X,Y,W,Z\in\Gamma
(TV_{\alpha}).$ The Levi-Civita connection $\nabla_{X}^{T}$ with respect to
the transverse K\"{a}hler metric $g^{T}$ is defined%
\[%
\begin{array}
[c]{rcl}%
\nabla_{X}^{T}Y & := & d\pi_{\alpha}(\nabla_{\widetilde{X}}\widetilde{Y}),\\
\widetilde{\nabla_{X}^{T}Y} & := & \nabla_{\widetilde{X}}\widetilde{Y}%
+g(JX,Y)\xi.
\end{array}
\]
and
\[%
\begin{array}
[c]{rcl}%
Rm^{T}(X,Y,Z,W) & = & Rm_{D}(\widetilde{X},\widetilde{Y},\widetilde{Z}%
,\widetilde{W})+g(J\widetilde{Y},\widetilde{W})g(J\widetilde{X},\widetilde{Z}%
)\\
&  & -g(J\widetilde{X},\widetilde{W})g(J\widetilde{Y},\widetilde{Z}%
)-2g(J\widetilde{X},\widetilde{Y})g(J\widetilde{Z},\widetilde{W}),\\
Ric^{T}(X,Z) & = & Ric_{D}(\widetilde{X},\widetilde{Z})+2g(\widetilde{X}%
,\widetilde{Z}).
\end{array}
\]
The transverse Ricci curvature $Ric^{T}$ of the Levi-Civita connection
$\nabla^{T}$ associated to $g^{T}$ is
\[
Ric^{T}=Ric+2g^{T}\text{ \textrm{and} }R^{T}=R+2n
\]
on $D=\ker\eta$ with
\begin{equation}%
\begin{array}
[c]{c}%
R_{i\overline{j}}^{T}=-\frac{\partial^{2}}{\partial z^{i}\partial\overline
{z}^{j}}\log\det(g_{\alpha\overline{\beta}}^{T}).
\end{array}
\label{2025}%
\end{equation}

\begin{definition}
For $\Phi$-invariant planes $\sigma_{1},$ $\sigma_{2}$ in $D_{p}\subset
T_{p}M$ such that $\sigma_{1}=\left\langle X,\Phi X\right\rangle $ and
$\sigma_{2}=\left\langle Y,\Phi Y\right\rangle $ with two perpendicular unit
vectors $X,$ $Y$, set
\[%
\begin{array}
[c]{c}%
U=\frac{1}{\sqrt{2}}(X-\sqrt{-1}\Phi X)\text{ \textrm{and\ }}V=\frac{1}%
{\sqrt{2}}(Y-\sqrt{-1}\Phi Y).
\end{array}
\]
We define \ the transverse holomorphic bisectional curvature
\[
Rm^{T}(\sigma_{1},\sigma_{2}):=Rm^{T}(U,\overline{U},V,\overline
{V}):=\left\langle Rm^{T}(X,Y)Y,X\right\rangle +\left\langle Rm^{T}(X,\Phi
Y)\Phi Y,X\right\rangle .
\]

\end{definition}

\begin{remark}
For an orthonormal basis $e_{1},\cdots,e_{2n}$ of $D_{p}\subset T_{p}M$ such
that $e_{n+i}=\Phi e_{i}$, $i=1,\cdots,n$ with $u_{i}=\frac{1}{\sqrt{2}}%
(e_{i}-\sqrt{-1}\Phi e_{i})$, then $\{u_{i}\}$ is a unitary basis. It follows
that
\[%
\begin{array}
[c]{c}%
Ric^{T}(u_{i},\overline{u}_{i})=\sum_{j}^{n}Rm^{T}(e_{i},e_{j},e_{j}%
,e_{i})+\sum_{j}^{n}Rm^{T}(e_{i},e_{n+j},e_{n+j},e_{i})=Ric^{T}(e_{i},e_{i}).
\end{array}
\]

\end{remark}

\begin{definition}
Let $\left(  M,\xi,\eta,\Phi\right)  $ be a complete noncompact Sasakian
$(2n+1)$-manifold. A $p$-form $\gamma$ is called basic if%
\[
i(\xi)\gamma=0\text{ \textrm{and} }\mathcal{L}_{\xi}\gamma=0.
\]

\end{definition}

Let $\Lambda_{B}^{p}$ be the sheaf of germs of basic $p$-forms and $\Omega
_{B}^{p}$ be the set of all global sections of $\Lambda_{B}^{p}$. It is easy
to check that $d\gamma$ is basic if $\gamma$ is basic. Set $d_{B}%
=d|_{\Omega_{B}^{p}}.$ Then%
\[
d_{B}:\Omega_{B}^{p}\rightarrow\Omega_{B}^{p+1}%
\]
Define%
\[%
\begin{array}
[c]{c}%
d_{B}:=\partial_{B}+\overline{\partial}_{B},\text{ }d_{B}^{c}:=\frac{1}%
{2}\sqrt{-1}(\overline{\partial}_{B}-\partial_{B}),
\end{array}
\]
then%
\[
d_{B}d_{B}^{c}=\sqrt{-1}\partial_{B}\overline{\partial}_{B},\text{ }d_{B}%
^{2}=(d_{B}^{c})^{2}=0.
\]
The basic Laplacian is defined by
\[
\Delta_{B}=d_{B}d_{B}^{\ast}+d_{B}^{\ast}d_{B}.
\]
Then we have the basic de Rham complex $(\Omega_{B}^{\ast},d_{B})$ and the
basic Dolbeault complex $(\Omega_{B}^{p,\ast},\overline{\partial}_{B})$ and
its cohomology group $H_{B}^{\ast}(M,\mathbb{R})$.

By the so-called
Type II deformations of Sasakian structures
$(M,\Phi,\eta,\xi,g_{\theta})$ fixing the $\xi$ and varying $\eta:$ For
$\varphi\in\Omega_{B}^{0}$, define
\[%
\begin{array}
[c]{c}%
\widetilde{\eta}=\eta+d_{B}^{c}\varphi,
\end{array}
\]
then
\[%
\begin{array}
[c]{c}%
d\widetilde{\eta}=d\eta+d_{B}d_{B}^{c}\varphi=d\eta+\sqrt{-1}\partial
_{B}\overline{\partial}_{B}\varphi\text{ \textrm{and}\ }\widetilde{\omega
}=\omega+\sqrt{-1}\partial_{B}\overline{\partial}_{B}\varphi.
\end{array}
\]
The the transverse K\"{a}hler metric
\[%
\begin{array}
[c]{c}%
\widetilde{g}^{T}=(g_{i\overline{j}}^{T}+\varphi_{i\overline{j}}%
)dz^{i}d\overline{z}^{j}.
\end{array}
\]
We consider the Sasakian Ricci flow
\begin{equation}%
\begin{array}
[c]{c}%
\frac{d}{dt}g^{T}(x,t)=-Ric^{T}(x,t).
\end{array}
\label{9}%
\end{equation}
It is equivalent to consider the parabolic Monge-Amp\`{e}re equation for a
basic function $\varphi$%
\begin{equation}%
\begin{array}
[c]{c}%
\frac{d}{dt}\varphi=\log\det(g_{\alpha\overline{\beta}}^{T}+\varphi
_{\alpha\overline{\beta}})-\log\det(g_{\alpha\overline{\beta}}^{T})-f
\end{array}
\label{2020A}%
\end{equation}
with some basic function $f$.

As in the paper of \cite[Appendix A.2.]{chlw}, for a compact Sasakian
$(2n+1)$-manifold $(M,\eta,\xi,\Phi)$, we define $D=\ker\eta$ to be the
holomorphic contact vector bundle of $TM$ such that
\[
TM=D\oplus\left\langle \xi\right\rangle =T^{1,0}(M)\oplus T^{0,1}%
(M)\oplus\left\langle \xi\right\rangle .
\]
Then its associated strictly pseudoconvex CR $(2n+1)$-manifold to be denoted
by $(M,T^{1,0}(M),\xi,\Phi)$ with the basic canonical bundle $K_{M}^{T}$. One
can consider $(K_{M}^{T},h)$ as a basic transverse holomorphic line bundle
over a Sasakian manifold $(M,\eta,\xi,\Phi)$ with the basic Hermitian metric
$h$. Note that if $g^{T}$ is a transverse K\"{a}hler metric on $M,$ then
$h=\det(g^{T})$ defines a basic Hermitian metric on the canonical bundle
$K_{M}^{T}$. The inverse $(K_{M}^{T})^{-1}$ of $K_{M}^{T}$ is sometimes called
the anti-canonical bundle. Its basic first Chern class $c_{1}^{B}((K_{M}%
^{T})^{-1})$ is called the basic first Chern class of $M$ and is often denoted
by $c_{1}^{B}(M).$ Then the basic first Chern class $c_{1}^{B}(M)$ is defined
to be $[Ric^{T}(\omega)]_{B}$ for any transverse K\"{a}hler metric $\omega$ on
a Sasakian manifold $M$.

As in the previous discussion in mind, we will define the basic first
Bott-Chern class $c_{1}^{BBC}(M)$.

\begin{definition}
There exists a unique transverse Chern connection $\nabla^{TC}$ which is
compatible with the transverse holomorphic structure $\nabla^{TC}%
=\overline{\partial}_{B}$ and $h_{i\overline{j}}^{T}$ such that
\[%
\begin{array}
[c]{c}%
\Gamma_{ij}^{k}=(h^{T})^{k\overline{l}}\partial_{i}h_{j\overline{l}}^{T}.
\end{array}
\]
The torsion $T_{ij}^{k}$ of $h_{i\overline{j}}^{T}$ is defined as usual
\[%
\begin{array}
[c]{c}%
T_{ij}^{k}=\Gamma_{ij}^{k}-\Gamma_{ji}^{k}.
\end{array}
\]

\begin{enumerate}
\item The transverse Chern curvature of $h_{i\overline{j}}^{T}$ is defined by%
\[%
\begin{array}
[c]{c}%
R_{\text{ }i\overline{j}k}^{T\text{ \ \ }l}=-\partial_{\overline{j}}%
\Gamma_{ik}^{l}.
\end{array}
\]

\item The transverse Chern-Ricci curvature is defined by
\[%
\begin{array}
[c]{c}%
R_{i\overline{j}}^{TC}=-\frac{\partial_{2}}{\partial z^{i}\partial\overline
{z}^{j}}\log\det(h_{\alpha\overline{\beta}}^{T}).
\end{array}
\]
That is the transverse Chern-Rcci form of $\omega_{h}$ is defined by
\[%
\begin{array}
[c]{c}%
Ric^{TC}(\omega_{h})=-\sqrt{-1}\partial_{B}\overline{\partial}_{B}\log
\det(h_{i\overline{j}}^{T})=\sqrt{-1}R_{i\overline{j}}^{TC}dz^{i}\wedge
d\overline{z}^{j}.
\end{array}
\]
Note that if $h_{i\overline{j}}^{T}$ is non-K\"{a}hler, $R_{i\overline{j}%
}^{TC}$ may not equal to $(h^{T})^{k\overline{l}}R_{\text{ }k\overline
{l}i\overline{j}}^{T\text{ \ \ }}.$

\item The basic Bott-Chern cohomology determined by the closed form
$Ric^{TC}(\omega_{h})$ is denoted%
\[%
\begin{array}
[c]{c}%
H_{BBC}^{1,1}(M,\mathbb{R})=\frac{\{\text{\textrm{closed real} }%
\mathrm{(1,1)}\text{\textrm{-forms}}\}}{\{\sqrt{-1}\partial_{B}\overline
{\partial}_{B}\phi,\phi\in C_{B}^{\infty}(M,\mathbb{R)}\}}%
\end{array}
\]
and the basic first Bott-Chern class is denoted by
\[
c_{1}^{BBC}(M)=c_{1}^{BBC}((K_{M}^{T})^{-1})=[Ric^{T}(\omega_{h})]_{BC}.
\]

\end{enumerate}
\end{definition}

In the following, we define the so-called Sasaki-Chern-Ricci flow for
transverse Hermitian metric $h_{i\overline{j}}^{T}$ on the normal bundle
$v(F_{\xi})$ with this transverse holomorphic structure with its associated
transverse K\"{a}hler $(1,1)$-form $\omega_{h}$.

The Chern-Ricci flow was first studied by Gill \cite{g} and then by
Tosatti-Weinkove \cite{tw}. Now we consider the Sasaki-Chern-Ricci flow on
$M\times\lbrack0,T)$ as
\begin{equation}%
\begin{array}
[c]{lll}%
\frac{d}{dt}\omega & = & -Ric^{TC}(\omega),\\
\omega(0) & = & \omega_{0}.
\end{array}
\label{1}%
\end{equation}
As long as the solution exists, the cohomology class $[\omega(t)]_{BC}$
evolves by%
\[%
\begin{array}
[c]{c}%
\frac{\partial}{\partial t}\left[  \omega(t)\right]  _{BC}=-c_{1}%
^{BBC}(M),\text{ }\left[  \omega(0)\right]  _{BBC}=\left[  \omega_{0}\right]
_{BBC},
\end{array}
\]
and solving this ordinary differential equation gives%
\[%
\begin{array}
[c]{c}%
\left[  \omega(t)\right]  _{BBC}=\left[  \omega_{0}\right]  _{BBC}%
-tc_{1}^{BBC}(M).
\end{array}
\]
We see that a necessary condition for the Sasaki-Chern-Ricci flow to exist for
$t>0$ such that%
\[%
\begin{array}
[c]{c}%
\left[  \omega_{0}\right]  _{BBC}-tc_{1}^{BBC}(M)>0.
\end{array}
\]
That is
\[%
\begin{array}
[c]{c}%
\omega(t)=\omega_{0}-tRic^{TC}(\omega_{0})+\sqrt{-1}\partial_{B}%
\overline{\partial}_{B}\varphi(t)>0.
\end{array}
\]
This necessary condition is sufficient. In fact we define
\[%
\begin{array}
[c]{c}%
T:=\sup\{t>0|\text{ }\exists\varphi\in C_{B}^{\infty}\text{ \textrm{such that}
}\omega_{0}-tRic^{TC}(\omega_{0})+\sqrt{-1}\partial_{B}\overline{\partial}%
_{B}\varphi(t)>0\}.
\end{array}
\]
As in the paper of Chang-Han-Lin-Wu \cite{chlw} and Tosatti-Weinkove
\cite{tw}, we have

\begin{proposition}
\label{P21}There exists a unique maximal solution $\omega(t)$ of the
Sasaki-Chern-Ricci flow (\ref{1}) on $M\times\lbrack0,T)$ for $t\in
\lbrack0,T)$.
\end{proposition}

That is, the Sasaki-Chern-Ricci flow can be rewritten as
\[%
\begin{array}
[c]{c}%
\left\{
\begin{array}
[c]{ccl}%
\frac{\partial}{\partial t}\omega & = & -Ric^{TC}(\omega_{0})+\sqrt
{-1}\partial_{B}\overline{\partial}_{B}\varphi(t),\\
\varphi(t) & = & \log\frac{\det h^{T}}{\det h_{0}^{T}},
\end{array}
\right.
\end{array}
\]
which is equivalent to the following parabolic Monge-Amp\`{e}re equation on
$[0,T)$
\begin{equation}%
\begin{array}
[c]{c}%
\left\{
\begin{array}
[c]{ccl}%
\frac{\partial}{\partial t}\varphi & = & \log\frac{(\omega_{0}-tRic^{TC}%
(\omega_{0})+\sqrt{-1}\partial_{B}\overline{\partial}_{B}\varphi)^{n}%
\wedge\eta_{0}}{\omega_{0}^{n}\wedge\eta_{0}};\\
\varphi(0) & = & 0.
\end{array}
\right.
\end{array}
\label{2025C}%
\end{equation}
with
\[%
\begin{array}
[c]{c}%
\varphi(x,t)=\int_{0}^{t}\log\frac{(\omega(x,s)^{n}}{\omega_{0}^{n}(x)}ds.
\end{array}
\]

\section{The Harnack Inequality}

In this section, we study the Li-Yau-Hamilton estimates and the Harnack
inequalities for Sasaki-Ricci flow
\begin{equation}%
\begin{array}
[c]{c}%
\frac{d}{dt}g_{\alpha\overline{\beta}}^{T}(x,t)=-R_{\alpha\overline{\beta}%
}^{T}(x,t)
\end{array}
\label{0}%
\end{equation}
on a Sasakian $(2n+1)$-manifold of bounded curvature and nonnegative
transverse holomorphic bisectional curvature.

\begin{definition}
(\cite{fow}) A complex vector field $X$ on a Sasakian manifold is called a
Hamilton holomorphic vector field if

\begin{enumerate}
\item $d\pi_{\alpha}(X)$ is a holomorphic vector field on $V_{\alpha},$

\item the complex valued Hamiltonian function $u_{X}\equiv\sqrt{-1}\eta(X)$
satisfies%
\[%
\begin{array}
[c]{c}%
\overline{\partial}_{B}u_{X}=-\frac{1}{2}d\eta(X,\cdot).
\end{array}
\]

\end{enumerate}
\end{definition}

If $(x,z^{1},\cdots,z^{n})$ is a foliation chart on a open cover $U_{\alpha},$
then $X$ can be written as%
\[%
\begin{array}
[c]{c}%
X=\eta(X)\frac{\partial}{\partial x}+\sum X^{\beta}\frac{\partial}{\partial
z^{\beta}}-\eta(\sum X^{\beta}\frac{\partial}{\partial z^{\beta}}%
)\frac{\partial}{\partial x},
\end{array}
\]
where $\frac{\partial}{\partial x}$ is the characteristic vector field and
$X^{\beta}=(g^{T})^{\beta\overline{\gamma}}\frac{\partial u}{\partial
\overline{z}^{\gamma}}$ are local holomorphic basic functions.

To formulate the Li-Yau-Hamilton quadric to the Sasaki-Ricci flow. Consider a
homothetically expanding transverse gradient Sasaki-Ricci soliton $g^{T}$
(\cite{cll}) satisfying%
\begin{equation}%
\begin{array}
[c]{c}%
R_{\alpha\overline{\beta}}^{T}+\frac{1}{t}g_{\alpha\overline{\beta}}%
^{T}=\nabla_{\alpha}^{T}X_{\overline{\beta}},\text{ }\nabla_{\overline{\alpha
}}^{T}X_{\overline{\beta}}=0
\end{array}
\label{01}%
\end{equation}
with $X^{\alpha}=(g^{T})^{\alpha\overline{\beta}}\frac{\partial u}%
{\partial\overline{z}^{\beta}}$ for some real-valued smooth holomorphic basic
function $u.$ Differentiating (\ref{01}) and commuting%
\[%
\begin{array}
[c]{c}%
\nabla_{\gamma}^{T}R_{\alpha\overline{\beta}}^{T}=\nabla_{\gamma}^{T}%
\nabla_{\alpha}^{T}X_{\overline{\beta}}=\nabla_{\gamma}^{T}\nabla
_{\overline{\beta}}^{T}X_{\alpha}=\nabla_{\gamma}^{T}\nabla_{\overline{\beta}%
}^{T}X_{\alpha}-\nabla_{\overline{\beta}}^{T}\nabla_{\gamma}^{T}X_{\alpha
}=-R_{\gamma\overline{\beta}\alpha\overline{\delta}}^{T}X^{\overline{\delta}},
\end{array}
\]
or%
\begin{equation}%
\begin{array}
[c]{c}%
\nabla_{\gamma}^{T}R_{\alpha\overline{\beta}}^{T}+R_{\alpha\overline{\beta
}\gamma\overline{\delta}}^{T}X^{\overline{\delta}}=0,
\end{array}
\label{02}%
\end{equation}
and%
\begin{equation}%
\begin{array}
[c]{c}%
\nabla_{\gamma}^{T}R_{\alpha\overline{\beta}}^{T}X^{\gamma}+R_{\alpha
\overline{\beta}\gamma\overline{\delta}}^{T}X^{\gamma}X^{\overline{\delta}}=0.
\end{array}
\label{03}%
\end{equation}
Differentiating (\ref{02}) again and using the first equation in (\ref{01}),
we get%
\begin{equation}%
\begin{array}
[c]{c}%
\nabla_{\overline{\sigma}}^{T}\nabla_{\gamma}^{T}R_{\alpha\overline{\beta}%
}^{T}+\nabla_{\overline{\sigma}}^{T}R_{\alpha\overline{\beta}\gamma
\overline{\delta}}^{T}X^{\overline{\sigma}}+R_{\alpha\overline{\beta}%
\gamma\overline{\epsilon}}^{T}R_{\epsilon\overline{\delta}}^{T}+\frac{1}%
{t}R_{\alpha\overline{\beta}\gamma\overline{\delta}}^{T}=0.
\end{array}
\label{04}%
\end{equation}
Taking the trace in (\ref{04}),
\begin{equation}%
\begin{array}
[c]{c}%
\Delta_{B}R_{\alpha\overline{\beta}}^{T}+\nabla_{\overline{\gamma}}%
^{T}R_{\alpha\overline{\beta}}^{T}X^{\overline{\gamma}}+R_{\alpha
\overline{\beta}\gamma\overline{\delta}}^{T}R_{\delta\overline{\gamma}}%
^{T}+\frac{1}{t}R_{\alpha\overline{\beta}}^{T}=0.
\end{array}
\label{05}%
\end{equation}
Symmetrizing by adding (\ref{02}) to (\ref{05}), we obtain%
\[%
\begin{array}
[c]{c}%
\Delta_{B}R_{\alpha\overline{\beta}}^{T}+\nabla_{\gamma}^{T}R_{\alpha
\overline{\beta}}^{T}X^{\gamma}+\nabla_{\overline{\gamma}}^{T}R_{\alpha
\overline{\beta}}^{T}X^{\overline{\gamma}}+R_{\alpha\overline{\beta}%
\gamma\overline{\delta}}^{T}R_{\delta\overline{\gamma}}^{T}+R_{\alpha
\overline{\beta}\gamma\overline{\delta}}^{T}X^{\gamma}X^{\overline{\delta}%
}+\frac{1}{t}R_{\alpha\overline{\beta}}^{T}=0.
\end{array}
\]

Next we consider the following Hamilton Harnack quantity
\[%
\begin{array}
[c]{c}%
Z_{\alpha\overline{\beta}}^{T}:=\frac{\partial}{\partial t}R_{\alpha
\overline{\beta}}^{T}+\nabla_{\gamma}^{T}R_{\alpha\overline{\beta}}%
^{T}X^{\gamma}+\nabla_{\overline{\gamma}}^{T}R_{\alpha\overline{\beta}}%
^{T}X^{\overline{\gamma}}+R_{\alpha\overline{\gamma}}^{T}R_{\gamma
\overline{\beta}}^{T}+R_{\alpha\overline{\beta}\gamma\overline{\delta}}%
^{T}X^{\gamma}X^{\overline{\delta}}+\frac{1}{t}R_{\alpha\overline{\beta}}^{T}.
\end{array}
\]
Now we can state the Li-Yau-Hamilton estimates and the Harnack inequalities
for the Sasaki-Ricci flow on a Sasakian $(2n+1)$-manifold with nonnegative
transverse holomorphic bisectional curvature.

\begin{theorem}
\label{Thm LYH}Let $g_{\alpha\overline{\beta}}^{T}(t),$ $t\in\lbrack0,T)$, be
a complete solution to the Sasaki-Ricci flow (\ref{0}) on a Sasakian
$(2n+1)$-manifold $M$ with bounded curvature and nonnegative transverse
holomorphic bisectional curvature. For any holomorphic Legendre tangent vector
$X$, let
\begin{equation}%
\begin{array}
[c]{c}%
Z_{\alpha\overline{\beta}}^{T}=\frac{\partial}{\partial t}R_{\alpha
\overline{\beta}}^{T}+\nabla_{\gamma}^{T}R_{\alpha\overline{\beta}}%
^{T}X^{\gamma}+\nabla_{\overline{\gamma}}^{T}R_{\alpha\overline{\beta}}%
^{T}X^{\overline{\gamma}}+R_{\alpha\overline{\gamma}}^{T}R_{\gamma
\overline{\beta}}^{T}+R_{\alpha\overline{\beta}\gamma\overline{\delta}}%
^{T}X^{\gamma}X^{\overline{\delta}}+\frac{1}{t}R_{\alpha\overline{\beta}}^{T}.
\end{array}
\label{07}%
\end{equation}
Then we have
\[%
\begin{array}
[c]{c}%
Z_{\alpha\overline{\beta}}^{T}W^{\alpha}W^{\overline{\beta}}\geq0
\end{array}
\]
for all holomorphic Legendre tangent vector $W$ and time $t\in(0,T).$
\end{theorem}

Theorem \ref{Thm LYH} works on the transverse K\"{a}hler structure. Then the
proof follows from essentially the same computation as in Cao \cite{c1} plus a
perturbation argument of Hamilton for strong \ maximum principle as in
\cite{h3}. We also refer to some related calculations as in \cite{he} and
\cite{cftw}.

\begin{proof}
Consider
\[%
\begin{array}
[c]{lll}%
Z=Z_{\alpha\overline{\beta}}^{T}W^{\alpha}W^{\overline{\beta}} & = &
\{\Delta_{B}R_{\alpha\overline{\beta}}^{T}+\nabla_{\gamma}^{T}R_{\alpha
\overline{\beta}}^{T}X^{\gamma}+\nabla_{\overline{\gamma}}^{T}R_{\alpha
\overline{\beta}}^{T}X^{\overline{\gamma}}\\
&  & \text{ }+R_{\alpha\overline{\beta}\gamma\overline{\delta}}^{T}%
R_{\delta\overline{\gamma}}^{T}+R_{\alpha\overline{\beta}\gamma\overline
{\delta}}^{T}X^{\gamma}X^{\overline{\delta}}+\frac{1}{t}R_{\alpha
\overline{\beta}}^{T}\}W^{\alpha}W^{\overline{\beta}}.
\end{array}
\]
Since the transverse holomorphic bisectional curvature is nonnegative it
implies that when $t$ is small $Z$ is nonnegative for any $x\in M$ and all
Legendre tangent vectors $X,$ $W\in T_{x}M$ with $W\neq0.$ We claim that this
remains true for all $t>0.$ In order to show this, we use the perturbation
argument as in \cite{h3}. From the work of Shi \cite{s1}, we can pick a basic
function $u(x)$ on $M$ so that $u(x)>1$, $u(x)\rightarrow\infty$ as $x$
$\rightarrow\infty$ and $|\Delta_{B}u|<C$ for a constant $C$. Then we put%
\[
\phi(x,t)=\epsilon e^{At}u(x),
\]
where $\epsilon>0$, $A>C$ are constants. Thus%
\begin{equation}%
\begin{array}
[c]{c}%
(\frac{\partial}{\partial t}-\Delta_{B})\phi=\epsilon e^{At}(Au-\Delta
_{B}u)>0.
\end{array}
\label{b}%
\end{equation}

Now we modify $Z$ by setting $\widetilde{Z}=Z+\phi$, then $\widetilde{Z}$ is
strictly positive everywhere at $t=0$. We claim that $\widetilde{Z}$ is
strictly positive everywhere for $t>0.$ Suppose this is not true, then there
is a first time $t_{0}>0$ such that $\widetilde{Z}$ vanishes for the first
time for some $X_{0}$ and $W_{0}\neq0$ in the Legendre tangent space of
$T_{x_{0}}M$ at some point $x_{0}\in M.$ Extend $X_{0}$ and $W_{0}$ to local
Legendre vector fields $X$ and $W$ in a neighborhood of $(x_{0},t_{0}),$ with%
\begin{equation}%
\begin{array}
[c]{c}%
X_{\gamma,\overline{\delta}}=X_{\overline{\delta},\gamma}=R_{\gamma
\overline{\delta}}^{T}+\frac{1}{t}g_{\alpha\overline{\beta}}^{T},\text{
}X_{\gamma,\delta}=X_{\overline{\delta},\overline{\gamma}}=0,\text{ }\\
(\frac{\partial}{\partial t}-\Delta_{B})X_{\gamma}=\frac{1}{2}R_{\gamma
\overline{\delta}}^{T}V_{\delta}-\frac{1}{t}X_{\gamma}%
\end{array}
\label{c}%
\end{equation}
and
\begin{equation}%
\begin{array}
[c]{c}%
W_{\overline{\alpha},\gamma}=W_{\overline{\alpha},\overline{\gamma}}=0,\text{
}(\frac{\partial}{\partial t}-\Delta_{B})W_{\overline{\alpha}}=0.
\end{array}
\label{d}%
\end{equation}
Following the same computation as \cite[Lemma 2.5]{c2}, the expression
$Z_{\alpha\overline{\beta}}^{T}$ satisfies the evolution equation%
\begin{equation}%
\begin{array}
[c]{c}%
(\frac{\partial}{\partial t}-\Delta_{B})Z_{\alpha\overline{\beta}}%
^{T}=R_{\alpha\overline{\beta}\gamma\overline{\delta}}^{T}Z_{\gamma
\overline{\delta}}^{T}+P_{\alpha\overline{\gamma}\delta}^{T}P_{\gamma
\overline{\beta}\overline{\delta}}^{T}-P_{\alpha\overline{\gamma}%
\overline{\delta}}^{T}P_{\gamma\overline{\beta}\delta}^{T}-\frac{1}%
{2}(R_{\alpha\overline{\gamma}}^{T}Z_{\gamma\overline{\beta}}^{T}%
+Z_{\alpha\overline{\gamma}}^{T}R_{\gamma\overline{\beta}}^{T})-\frac{2}%
{t}Z_{\alpha\overline{\beta}}^{T}%
\end{array}
\label{a}%
\end{equation}
at $(x_{0},t_{0}),$ where
\[%
\begin{array}
[c]{c}%
P_{\alpha\overline{\beta}\gamma}^{T}=\nabla_{\gamma}^{T}R_{\alpha
\overline{\beta}}^{T}+R_{\alpha\overline{\beta}\gamma\overline{\delta}}%
^{T}X^{\overline{\delta}}\text{ }\mathrm{and}\text{ }P_{\alpha\overline{\beta
}\overline{\delta}}^{T}=\nabla_{\overline{\delta}}^{T}R_{\alpha\overline
{\beta}}^{T}+R_{\alpha\overline{\beta}\gamma\overline{\delta}}^{T}X^{\gamma}.
\end{array}
\]
By applying (\ref{b}), (\ref{d}) and using (\ref{a}), at $(x_{0},t_{0}),$ we
obtain
\begin{equation}%
\begin{array}
[c]{ll}
& (\frac{\partial}{\partial t}-\Delta_{B})\widetilde{Z}\\
= & \{(\frac{\partial}{\partial t}-\Delta_{B})Z_{\alpha\overline{\beta}}%
^{T}\}W^{\alpha}W^{\overline{\beta}}+Z_{\alpha\overline{\beta}}^{T}W^{\alpha
}(\frac{\partial}{\partial t}-\Delta_{B})W^{\overline{\beta}}+Z_{\alpha
\overline{\beta}}^{T}W^{\overline{\beta}}(\frac{\partial}{\partial t}%
-\Delta_{B})W^{\alpha}\\
& -\nabla_{\gamma}^{T}Z_{\alpha\overline{\beta}}^{T}\nabla_{\overline{\gamma}%
}^{T}(W^{\alpha}W^{\overline{\beta}})-\nabla_{\overline{\gamma}}^{T}%
Z_{\alpha\overline{\beta}}^{T}\nabla_{\gamma}^{T}(W^{\alpha}W^{\overline
{\beta}})+(\frac{\partial}{\partial t}-\Delta_{B})\phi\\
= & R_{\alpha\overline{\beta}\gamma\overline{\delta}}^{T}Z_{\gamma
\overline{\delta}}^{T}W^{\alpha}W^{\overline{\beta}}+P_{\alpha\overline
{\gamma}\delta}^{T}W^{\alpha}P_{\gamma\overline{\beta}\overline{\delta}}%
^{T}W^{\overline{\beta}}-P_{\alpha\overline{\gamma}\overline{\delta}}%
^{T}W^{\alpha}P_{\gamma\overline{\beta}\delta}^{T}W^{\overline{\beta}}\\
& -\frac{1}{2}(R_{\alpha\overline{\gamma}}^{T}Z_{\gamma\overline{\beta}}%
^{T}+Z_{\alpha\overline{\gamma}}^{T}R_{\gamma\overline{\beta}}^{T})W^{\alpha
}W^{\overline{\beta}}-\frac{2}{t}\widetilde{Z}+\frac{2}{t}\phi+\epsilon
e^{At}(Au-\Delta_{B}u).
\end{array}
\label{e}%
\end{equation}
Since $W$ is a zero Legendre eigenvector of $Z_{\alpha\overline{\beta}}^{T}$
at $(x_{0},t_{0}),$ the first two terms in the last line of the above equality
disappear. Set
\[%
\begin{array}
[c]{c}%
M_{\gamma\overline{\delta}}^{T}=P_{\gamma\overline{\delta}\alpha}^{T}%
W^{\alpha}=(\nabla_{\gamma}^{T}R_{\alpha\overline{\delta}}^{T}+R_{\alpha
\overline{\beta}\gamma\overline{\delta}}^{T}X^{\overline{\beta}})W^{\alpha}%
\end{array}
\]
and%
\[%
\begin{array}
[c]{c}%
M_{\gamma\delta}^{T}=P_{\gamma\overline{\beta}\delta}^{T}W^{\overline{\beta}%
}=(\nabla_{\delta}^{T}R_{\gamma\overline{\beta}}^{T}+R_{\alpha\overline{\beta
}\gamma\overline{\delta}}^{T}X^{\alpha})W^{\overline{\beta}}.
\end{array}
\]
It can be showed as \cite[Proposition 4.1]{c1} that, at $(x_{0},t_{0}),$%
\[%
\begin{array}
[c]{c}%
R_{\alpha\overline{\beta}\gamma\overline{\delta}}^{T}Z_{\gamma\overline
{\delta}}^{T}W^{\alpha}W^{\overline{\beta}}+|M_{\gamma\overline{\delta}}%
^{T}|^{2}-|M_{\gamma\delta}^{T}|^{2}\geq0.
\end{array}
\]
This would imply form (\ref{e}) that, at $(x_{0},t_{0}),$
\[%
\begin{array}
[c]{l}%
(\frac{\partial}{\partial t}-\Delta_{B})\widetilde{Z}\geq\frac{2}{t_{0}}%
\phi+\epsilon e^{At_{0}}(Au-\Delta_{B}u)>0
\end{array}
\]
which contradicts the fact that $\widetilde{Z}$ achieves a minimum at
$(x_{0},t_{0})$ by the strong maximum principle. Therefore $\widetilde{Z}$ is
strictly positive everywhere for $t>0$ and it implies that $Z\geq0$ for
all$\ t>0$ by letting $\epsilon\rightarrow0$. This proves Theorem
\ref{Thm LYH}.
\end{proof}

Taking the trace in Theorem \ref{Thm LYH}, we have the following Harnack
estimate for the transverse scalar curvature $R^{T}$.

\begin{corollary}
\label{Cor}Under the assumption of Theorem \ref{Thm LYH}, the transverse
scalar curvature $R^{T}$ satisfies the estimate
\begin{equation}%
\begin{array}
[c]{c}%
\frac{\partial R^{T}}{\partial t}+\nabla_{\alpha}^{T}R^{T}X^{\alpha}%
+\nabla_{\overline{\alpha}}^{T}R^{T}X^{\overline{\alpha}}+R_{\alpha
\overline{\beta}}^{T}X^{\alpha}X^{\overline{\beta}}+\frac{R^{T}}{t}\geq0
\end{array}
\label{08}%
\end{equation}
for all holomorphic Legendre tangent vector $X$ and $t>0.$ In particular,%
\begin{equation}%
\begin{array}
[c]{c}%
\frac{\partial R^{T}}{\partial t}-\frac{|\nabla^{T}R^{T}|^{2}}{R^{T}}%
+\frac{R^{T}}{t}\geq0.
\end{array}
\label{09}%
\end{equation}

\end{corollary}

\begin{proof}
The first inequality (\ref{08}) follows by taking the trace in Theorem
\ref{Thm LYH}. The second inequality (\ref{09}) follows by taking
$X=-\nabla^{T}\ln R^{T}$ in (\ref{08}) and observing $R_{\alpha\overline
{\beta}}^{T}\leq R^{T}g_{\alpha\overline{\beta}}^{T},$ since $R_{\alpha
\overline{\beta}}^{T}\geq0.$
\end{proof}

As a consequence of Corollary \ref{Cor}, we obtain the following Harnack
inequality for the transverse scalar curvature $R^{T}.$

\begin{corollary}
Let $g_{\alpha\overline{\beta}}^{T}(t)$ be a complete solution to the
Sasaki-Ricci flow (\ref{0}) on a Sasakian $(2n+1)$-manifold $M$ with bounded
curvature and nonnegative transverse holomorphic bisectional curvature. Then
for any points $x_{1},$ $x_{2}\in M,$ and $0<t_{1}<t_{2},$ we have%
\[%
\begin{array}
[c]{c}%
R^{T}(x_{1},t_{1})\leq R^{T}(x_{2},t_{2})\exp\left(  \frac{d_{t_{1}}^{T}%
(x_{1},x_{2})^{2}}{4(t_{2}-t_{1})}\right)  .
\end{array}
\]
Here $d_{t_{1}}^{T}(x_{1},x_{2})$ denotes the transverse distance between
$x_{1}$ and $x_{2}$ with respect to $g_{\alpha\overline{\beta}}^{T}(t_{1}).$
\end{corollary}

\begin{proof}
Take a Legendrian geodesic path $\gamma(\tau),$ $\tau\in\lbrack t_{1},t_{2}],$
from $x_{1}$ to $x_{2}$ at time $t_{1}$ with constant velocity $d_{t_{1}}%
^{T}(x_{1},x_{2})/(t_{2}-t_{1}).$ Consider the space-time path $\eta
(\tau)=(\gamma(\tau),\tau),$ $\tau\in\lbrack t_{1},t_{2}].$ We compute%
\[%
\begin{array}
[c]{lll}%
\ln\frac{R^{T}(x_{2},t_{2})}{R^{T}(x_{1},t_{1})} & = & \int_{t_{1}}^{t_{2}%
}\frac{d}{d\tau}\ln R^{T}(\gamma(\tau),\tau)d\tau\\
& = & \int_{t_{1}}^{t_{2}}\frac{1}{R^{T}}(\frac{\partial R^{T}}{\partial\tau
}+\langle\nabla^{T}R^{T},\frac{d\gamma}{d\tau}\rangle)d\tau\\
& \geq & \int_{t_{1}}^{t_{2}}\frac{\partial}{\partial\tau}\ln R^{T}%
-|\nabla^{T}\ln R^{T}|_{g^{T}(\tau)}^{2}-\frac{1}{4}|\frac{d\gamma}{d\tau
}|_{g^{T}(\tau)}^{2}d\tau.
\end{array}
\]
Then, by the Li-Yau-Hamilton estimate (\ref{09}) for $R^{T}$ in Corollary
\ref{Cor} and the fact that the metric is shrinking since the transverse Ricci
curvature is nonnegative, we have
\[%
\begin{array}
[c]{l}%
\ln\frac{R^{T}(x_{2},t_{2})}{R^{T}(x_{1},t_{1})}\geq\int_{t_{1}}^{t_{2}%
}(-\frac{1}{\tau}-\frac{1}{4}|\frac{d\gamma}{d\tau}|_{g^{T}(\tau)}^{2}%
)d\tau=\ln\frac{t_{1}}{t_{2}}-\frac{d_{t_{1}}^{T}(x_{1},x_{2})^{2}}%
{4(t_{2}-t_{1})}.
\end{array}
\]
The desired Harnack inequality follows by taking exponential of the above inequality.
\end{proof}

If our solution is defined for $-\infty<t<0$, we can drop the term $\frac
{1}{t}R_{\alpha\overline{\beta}}^{T}$ in Theorem \ref{Thm LYH}. Then we have

\begin{theorem}
\label{Thm2}Let $g_{\alpha\overline{\beta}}^{T}(t),$ $t\in(-\infty,T)$, be a
complete solution to the Sasaki-Ricci flow (\ref{0}) on a Sasakian
$(2n+1)$-manifold $M$ with bounded curvature and nonnegative transverse
holomorphic bisectional curvature. For any holomorphic Legendre tangent vector
$X$, let
\[%
\begin{array}
[c]{c}%
Q_{\alpha\overline{\beta}}^{T}=\frac{\partial}{\partial t}R_{\alpha
\overline{\beta}}^{T}+\nabla_{\gamma}^{T}R_{\alpha\overline{\beta}}%
^{T}X^{\gamma}+\nabla_{\overline{\gamma}}^{T}R_{\alpha\overline{\beta}}%
^{T}X^{\overline{\gamma}}+R_{\alpha\overline{\gamma}}^{T}R_{\gamma
\overline{\beta}}^{T}+R_{\alpha\overline{\beta}\gamma\overline{\delta}}%
^{T}X^{\gamma}X^{\overline{\delta}}.
\end{array}
\]
Then we have
\[%
\begin{array}
[c]{c}%
Q_{\alpha\overline{\beta}}^{T}W^{\alpha}W^{\overline{\beta}}\geq0
\end{array}
\]
for all holomorphic Legendre tangent vector $W$ and time $t\in(-\infty,T).$
\end{theorem}

Taking the trace again in Theorem \ref{Thm2}, we have

\begin{corollary}
Under the assumption of Theorem \ref{Thm2}, the transverse scalar curvature
$R^{T}$ satisfies the estimate
\[%
\begin{array}
[c]{c}%
\frac{\partial R^{T}}{\partial t}+\nabla_{\alpha}^{T}R^{T}X^{\alpha}%
+\nabla_{\overline{\alpha}}^{T}R^{T}X^{\overline{\alpha}}+R_{\alpha
\overline{\beta}}^{T}X^{\alpha}X^{\overline{\beta}}\geq0
\end{array}
\]
for all holomorphic Legendre tangent vector $X.$
\end{corollary}

\begin{theorem}
Let $g_{\alpha\overline{\beta}}^{T}(t)$ be a complete solution to the
Sasaki-Ricci flow (\ref{0}) on a simply connected non-compact Sasakian
$(2n+1)$-manifold $M$ defined for $-\infty<t<\infty$ with uniformly bounded
curvature, nonnegative transverse holomorphic bisectional curvature and
positive transverse Ricci curvature where the transverse scalar curvature
$R^{T}$ assumes its maximum in space-time. Then $g_{\alpha\overline{\beta}%
}^{T}$ is necessarily a gradient Sasaki-Ricci soliton.
\end{theorem}

\begin{proof}
Let $g_{\alpha\overline{\beta}}^{T}(t)$ be a complete solution to equation
(\ref{0}) defined for $-\infty<t<\infty$ with bounded curvature, nonnegative
transverse holomorphic bisectional curvature and positive transverse Ricci
curvature where the transverse scalar curvature $R^{T}$ assumes its maximum at
a point $(x_{0},t_{0})$ in space time. Then at this point, $\frac{\partial
R^{T}}{\partial t}=0$ and the quadratic form
\[%
\begin{array}
[c]{c}%
Q=(g^{T})^{\alpha\overline{\beta}}Q_{\alpha\overline{\beta}}^{T}%
=\frac{\partial R^{T}}{\partial t}+\nabla_{\alpha}^{T}R^{T}X^{\alpha}%
+\nabla_{\overline{\alpha}}^{T}R^{T}X^{\overline{\alpha}}+R_{\alpha
\overline{\beta}}^{T}X^{\alpha}X^{\overline{\beta}}=0
\end{array}
\]
in the direction $X=0.$ By the strong maximum principle we know that at any
earlier time, there is a $X$ such that $Q=0$ in the direction $X$. In fact, by
considering the first variation of $Q$ in $X$, null vectors must satisfy the
equation
\begin{equation}%
\begin{array}
[c]{c}%
\nabla_{\alpha}^{T}R^{T}+R_{\alpha\overline{\beta}}^{T}X^{\overline{\beta}}=0.
\end{array}
\label{t1}%
\end{equation}
Since the transverse Ricci tensor is positive, we see that such a null vector
$X$ is unique and varies smoothly in space-time. Hence it gives rise to a
time-dependent smooth section of the holomorphic Legendre tangent bundle. For
such a null vector $X,$ we have $Q=0.$ It follows from (\ref{t1}) that
\begin{equation}%
\begin{array}
[c]{c}%
\frac{\partial R^{T}}{\partial t}+\nabla_{\overline{\alpha}}^{T}%
R^{T}X^{\overline{\alpha}}=0.
\end{array}
\label{t2}%
\end{equation}
Moreover, since $Q_{\alpha\overline{\beta}}^{T}\geq0$ and its trace
$Q=(g^{T})^{\alpha\overline{\beta}}Q_{\alpha\overline{\beta}}^{T}=0$, we have
\[%
\begin{array}
[c]{c}%
Q_{\alpha\overline{\beta}}^{T}=\frac{\partial}{\partial t}R_{\alpha
\overline{\beta}}^{T}+\nabla_{\gamma}^{T}R_{\alpha\overline{\beta}}%
^{T}X^{\gamma}+\nabla_{\overline{\gamma}}^{T}R_{\alpha\overline{\beta}}%
^{T}X^{\overline{\gamma}}+R_{\alpha\overline{\gamma}}^{T}R_{\gamma
\overline{\beta}}^{T}+R_{\alpha\overline{\beta}\gamma\overline{\delta}}%
^{T}X^{\gamma}X^{\overline{\delta}}=0.
\end{array}
\]
Again from the first variation of $Q_{\alpha\overline{\beta}}^{T}$ in $X$, it
follows that
\begin{equation}%
\begin{array}
[c]{c}%
\nabla_{\gamma}^{T}R_{\alpha\overline{\beta}}^{T}+R_{\alpha\overline{\beta
}\gamma\overline{\delta}}^{T}X^{\overline{\delta}}=0,
\end{array}
\label{t3}%
\end{equation}
and hence
\begin{equation}%
\begin{array}
[c]{c}%
\frac{\partial}{\partial t}R_{\alpha\overline{\beta}}^{T}+R_{\alpha
\overline{\gamma}}^{T}R_{\gamma\overline{\beta}}^{T}+\nabla_{\overline{\gamma
}}^{T}R_{\alpha\overline{\beta}}^{T}X^{\overline{\gamma}}=0.
\end{array}
\label{t4}%
\end{equation}
Note that equations (\ref{t3}) and (\ref{t4}) are the first order and the
second order equations satisfied by a gradient Sasaki-Ricci soliton. Now we
show that these identities indeed imply that $X$ is a holomorphic Legendre
gradient vector field and the solution $g_{\alpha\overline{\beta}}^{T}$ is a
gradient soliton.

First, by applying the operator $(\frac{\partial}{\partial t}-\Delta_{B})$ to
(\ref{t1}) and using the evolution equations for the transverse Ricci
curvature and the transverse scalar curvature and (\ref{t3}), we get
\begin{equation}%
\begin{array}
[c]{lll}%
0 & = & (\frac{\partial}{\partial t}-\Delta_{B})(\nabla_{\alpha}^{T}%
R^{T}+R_{\alpha\overline{\beta}}^{T}X^{\overline{\beta}})\\
& = & \nabla_{\alpha}^{T}R_{\gamma\overline{\delta}}^{T}(R_{\delta
\overline{\gamma}}^{T}-\nabla_{\overline{\gamma}}^{T}X^{\overline{\delta}%
})+\frac{1}{2}R_{\alpha\overline{\beta}}^{T}\nabla_{\beta}^{T}R^{T}%
+R_{\alpha\overline{\beta}}^{T}(\frac{\partial}{\partial t}-\Delta
_{B})X^{\overline{\beta}}-\nabla_{\overline{\gamma}}^{T}R_{\alpha
\overline{\beta}}^{T}\nabla_{\gamma}^{T}X^{\overline{\beta}}\\
& = & \nabla_{\alpha}^{T}R_{\gamma\overline{\delta}}^{T}(R_{\delta
\overline{\gamma}}^{T}-\nabla_{\overline{\gamma}}^{T}X^{\overline{\delta}%
})-\nabla_{\overline{\gamma}}^{T}R_{\alpha\overline{\beta}}^{T}\nabla_{\gamma
}^{T}X^{\overline{\beta}},
\end{array}
\label{t5}%
\end{equation}
here $(\frac{\partial}{\partial t}-\Delta_{B})X_{\beta}=\frac{1}{2}%
R_{\beta\overline{\gamma}}^{T}X^{\overline{\gamma}}=-\frac{1}{2}\nabla_{\beta
}^{T}R^{T}$ at a point. Next, we apply $(\frac{\partial}{\partial t}%
-\Delta_{B})$ to (\ref{t2}). From (\ref{t4}), we have
\begin{equation}%
\begin{array}
[c]{lll}%
0 & = & (\frac{\partial}{\partial t}-\Delta_{B})(\frac{\partial R^{T}%
}{\partial t}+\nabla_{\overline{\alpha}}^{T}R^{T}X^{\overline{\alpha}})\\
& = & 2\nabla_{\overline{\beta}}^{T}\nabla_{\alpha}^{T}R^{T}R_{\beta
\overline{\alpha}}^{T}+R_{\alpha\overline{\beta}}^{T}R_{\beta\overline{\gamma
}}^{T}R_{\gamma\overline{\alpha}}^{T}+\nabla_{\overline{\alpha}}^{T}%
R_{\beta\overline{\gamma}}^{T}R_{\gamma\overline{\beta}}^{T}X^{\overline
{\alpha}}-\frac{1}{2}R_{\alpha\overline{\beta}}^{T}X^{\overline{\beta}}%
\nabla_{\overline{\alpha}}^{T}R^{T}\\
&  & +\nabla_{\overline{\alpha}}^{T}R^{T}(\frac{\partial}{\partial t}%
-\Delta_{B})X^{\overline{\alpha}}-\nabla_{\overline{\beta}}^{T}\nabla_{\alpha
}^{T}R^{T}\nabla_{\overline{\alpha}}^{T}X^{\overline{\beta}}-\nabla
_{\overline{\beta}}^{T}\nabla_{\overline{\alpha}}^{T}R^{T}\nabla_{\alpha}%
^{T}X^{\overline{\beta}}\\
& = & \nabla_{\overline{\beta}}^{T}\nabla_{\alpha}^{T}R^{T}(2R_{\beta
\overline{\alpha}}^{T}-\nabla_{\overline{\alpha}}^{T}X^{\overline{\beta}%
})+R_{\alpha\overline{\beta}}^{T}R_{\beta\overline{\gamma}}^{T}R_{\gamma
\overline{\alpha}}^{T}+\nabla_{\overline{\alpha}}^{T}R_{\beta\overline{\gamma
}}^{T}R_{\gamma\overline{\beta}}^{T}X^{\overline{\alpha}}\\
&  & -\nabla_{\overline{\beta}}^{T}\nabla_{\overline{\alpha}}^{T}R^{T}%
\nabla_{\beta}^{T}X^{\overline{\alpha}}.
\end{array}
\label{t6}%
\end{equation}

From (\ref{t1}), one get
\[%
\begin{array}
[c]{l}%
\nabla_{\overline{\beta}}^{T}\nabla_{\alpha}^{T}R^{T}=-R_{\alpha
\overline{\gamma}}^{T}\nabla_{\overline{\beta}}^{T}X^{\overline{\gamma}%
}-\nabla_{\overline{\beta}}^{T}R_{\alpha\overline{\gamma}}^{T}X^{\overline
{\gamma}}=-R_{\gamma\overline{\beta}}^{T}\nabla_{\alpha}^{T}X^{\gamma}%
-\nabla_{\gamma}^{T}R_{\alpha\overline{\beta}}^{T}X^{\gamma},
\end{array}
\]
then substitute this equation into (\ref{t6}) and use (\ref{t5}), we obtain%
\begin{equation}%
\begin{array}
[c]{l}%
R_{\alpha\overline{\beta}}^{T}(R_{\beta\overline{\gamma}}^{T}-\nabla
_{\overline{\gamma}}^{T}X^{\overline{\beta}})(R_{\gamma\overline{\alpha}}%
^{T}-\nabla_{\gamma}^{T}X^{\overline{\alpha}})+R_{\alpha\overline{\beta}}%
^{T}\nabla_{\gamma}^{T}X^{\overline{\beta}}\nabla_{\overline{\gamma}}%
^{T}X^{\alpha}=0.
\end{array}
\label{t7}%
\end{equation}
Now we diagonalize the transverse Ricci tensor so that $R_{\alpha
\overline{\beta}}^{T}=R_{\alpha\overline{\alpha}}^{T}\delta_{\alpha\beta}.$
Thus (\ref{t7}) will become
\[%
\begin{array}
[c]{l}%
R_{\alpha\overline{\alpha}}^{T}[|R_{\alpha\overline{\beta}}^{T}-\nabla
_{\overline{\beta}}^{T}X^{\overline{\alpha}}|^{2}+|\nabla_{\beta}%
^{T}X^{\overline{\alpha}}|^{2}]=0.
\end{array}
\]
Since the transverse Ricci curvature is positive we get
\begin{equation}%
\begin{array}
[c]{l}%
\nabla_{\beta}^{T}X^{\overline{\alpha}}=\nabla_{\overline{\beta}}^{T}%
X^{\alpha}=0\text{ \ \textrm{and} \ }R_{\alpha\overline{\beta}}^{T}%
=\nabla_{\overline{\beta}}^{T}X^{\overline{\alpha}}=\nabla_{\alpha}%
^{T}X^{\beta}.
\end{array}
\label{t8}%
\end{equation}
Therefore $X$ is a holomorphic Legendre vector field and the solution metric
$g_{\alpha\overline{\beta}}^{T}$ is a soliton. Furthermore, the second
identity in (\ref{t8}) and since $M$ is simply connected imply that the vector
field $X$ is the gradient of some basic function $f$ on $M$. So our soliton
$g_{\alpha\overline{\beta}}^{T}$ is indeed a gradient soliton.
\end{proof}

\section{The
Long-Time Solution
}

In this section, we will show Theorem \ref{T1A}.

(I) For the long-time existence, we first prove that

\begin{lemma}
\label{22}Let $(M^{2n+1},J,\theta)$ be a compact noncompact Sasakian
$(2n+1)$-manifold with nonnegative transverse bisectional curvature and
$n\geq2$. Suppose that there exist positive constants $C_{1}$, $C_{2}$ and
$0<\varepsilon\leq1$ such for any $x_{0}$,%
\[%
\begin{array}
[c]{c}%
R^{T}(x,0)\leq C_{1},\text{ }x\in M,\\
\frac{1}{\mathrm{Vol}(B_{\xi,g^{T}(x,0)}(x_{0},r))}\int_{B_{\xi,g^{T}%
(x,0)}(x_{0},r)}R^{T}(x,0)\leq\frac{C_{2}}{1+r^{1+\varepsilon}},\text{ }0\leq
r<+\infty,
\end{array}
\]
where $B_{\xi,g^{T}(x,0)}(x_{0},r)$ is the transverse ball with respect to the
transverse metric $g^{T}(x,0)$. Then the solution of (\ref{9}) satisfies the
following
\begin{equation}%
\begin{array}
[c]{c}%
F(x,t)\geq-C(1+t)^{\frac{1-\varepsilon}{1+\varepsilon}}%
\end{array}
\label{32}%
\end{equation}
on $M\times\lbrack0,t_{\max})$, where $0<C<+\infty$ is a constant depending
only on $n,$ $\varepsilon,$ $C_{1}$ and $C_{2}$. In particular, if
$\varepsilon=1$, we have
\[
F(x,t)\geq-C.
\]

\end{lemma}

\begin{proof}
We define a function $F(x,t)$ on $M\times\lbrack0,t_{\max})$ as follows
\[%
\begin{array}
[c]{c}%
F(x,t)=\log\frac{\text{\textrm{det}}(g_{i\bar{j}}^{T}(x,t))}%
{\text{\textrm{det}}(g_{i\bar{j}}^{T}(x,0))}.
\end{array}
\]
So we can obtain the following
\begin{equation}%
\begin{array}
[c]{c}%
\frac{\partial F(x,t)}{\partial t}=g^{T{i\bar{j}}}\frac{\partial}{\partial
t}g_{i\bar{j}}^{T}=-R^{T}%
\end{array}
\label{8}%
\end{equation}
Because of the nonnegativity of the transverse bisectional curvature of
$g_{i\bar{j}}^{T}$, we can know that $F(x,t)$ is non-increasing in $t$ and
$F(x,0)=0$. And by the equation (\ref{9}), we can obtain that
\[%
\begin{array}
[c]{c}%
g_{i\bar{j}}^{T}(x,t)\leq g_{i\bar{j}}^{T}(x,0)
\end{array}
\]
on $M$. Now we compute%
\begin{equation}%
\begin{array}
[c]{ll}
& e^{F(x,t)}R^{T}(x,t)\\
= & \frac{\text{\textrm{det}}(g_{i\bar{j}}^{T}(x,t))}{\text{\textrm{det}%
}(g_{i\bar{t}}^{T}(x,0))}g^{Ti\bar{j}}(x,t)R_{i\bar{j}}^{T}(x,t)\\
\leq & g^{Ti\bar{j}}(x,0)R_{i\bar{j}}^{T}(x,t)\\
= & g^{Ti\bar{j}}(x,0)(R_{i\bar{j}}^{T}(x,t)-R_{i\bar{j}}^{T}(x,0)))+g^{Ti\bar
{j}}(x,0)R_{i\bar{j}}^{T}(x,0)\\
= & -g^{Ti\bar{j}}(x,0)\left(  \frac{\partial^{2}}{\partial z^{i}\partial
\bar{z}^{j}}\log\text{\textrm{det}}(g_{i\bar{j}}^{T}(x,t))-\frac{\partial^{2}%
}{\partial z^{i}\partial\bar{z}^{j}}\log\text{\textrm{det}}(g_{i\bar{j}}%
^{T}(x,0)\right)  +R^{T}(x,0)\\
= & -g^{Ti\bar{j}}(x,0)\frac{\partial^{2}F(x,t)}{\partial z^{i}\partial\bar
{z}^{j}}+R^{T}(x,0)\\
= & -\Delta_{b,g^{T}(x,0)}F(x,t)+R^{T}(x,0),
\end{array}
\label{10}%
\end{equation}
where $\Delta_{B,g^{T}(x,0)}$ is the basic Laplace operator with respect to
the initial metric $g_{i\bar{j}}^{T}(x,0)$. From (\ref{8}) and (\ref{10}), we
can obtain the following
\[%
\begin{array}
[c]{c}%
e^{F(x,t)}\frac{\partial F(x,t)}{\partial t}\geq\Delta_{B,g^{T}(x,0)}%
F(x,t)-R^{T}(x,0),
\end{array}
\]
and
\[%
\begin{array}
[c]{c}%
\Delta_{B,g^{T}(x,0)}F(x,t)\leq R^{T}(x,0)
\end{array}
\]
on $M\times\lbrack0,t_{\max})$.

Since the transverse bisectional curvature of $g_{i\bar{j}}^{T}(x,0)$ is
nonnegative, we can know that the transverse Ricci curvature is nonnegative.
Let $H(x,y,t)$ be the heat kernel of the CR heat equation on $M$ with respect
to $g_{i\bar{j}}^{T}(x,0)$. We define the Green function
\[%
\begin{array}
[c]{c}%
G_{g_{i\bar{j}}^{T}(x,0)}(x,y)=\int_{0}^{\infty}H(x,y,t)dt,
\end{array}
\]
if the integral on the right side converges. Then by \cite[Theorem 2.2.]{cchl}
and the subgradient estimate in \cite{cklt}, we know that for $x,y\in
M^{2n+1}$,
\[%
\begin{array}
[c]{c}%
\sigma^{-1}\frac{r_{0}^{2}(x,y)}{Vol(B_{0}(x,r_{0}(x,y)))}\leq G_{g_{i\bar{j}%
}^{T}(x,0)}(x,y)\leq\sigma\frac{r_{0}^{2}(x,y)}{Vol(B_{0}(x,r_{0}(x,y)))}%
\end{array}
\]
and
\[%
\begin{array}
[c]{c}%
|\nabla_{B,g_{i\bar{j}}^{T}(x,0)}G_{g_{i\bar{j}}^{T}(x,0)}(x,y)|^{2}\leq
\frac{C_{3}r_{0}(x,y)}{\mathrm{Vol}(B_{0}(x,r_{0}(x,y)))},
\end{array}
\]
where $B_{0}(x,r_{0}(x,y)))$ is the geodesic ball of radius and centered at
$x_{0}$ with respect to $g_{i\bar{j}}(x,0)$, and $r_{0}(x,y)$ is the distance
between $x$ and $y$ with respect to $g_{i\bar{j}}(x,0)$, $\sigma$ and $C_{3}$
are positive constants depending only on $n$. We define
\[%
\begin{array}
[c]{c}%
\Omega_{\alpha}=\{y\in M^{2n+1}:G_{g_{i\bar{j}}^{T}(x,0)}(x,y)>\alpha\}
\end{array}
\]
for any $\alpha>0$. Since $F(x,t)$ is a basic function on $M^{2n+1}$ and by
using Green formula for any $x_{0}\in M^{2n+1}$ and any $t\in\lbrack0,t_{\max
})$, we have the following%
\begin{equation}%
\begin{array}
[c]{lll}%
F(x_{0},t) & = & \int_{\Omega_{\alpha}}(\alpha-G_{g_{i\bar{j}}^{T}(x,0)}%
(x_{0},y))\Delta_{b,g_{i\bar{j}}^{T}(x,0)}F(x,t)dv(y)\\
&  & -\int_{\partial\Omega_{\alpha}}F(y,t)\frac{\partial G_{g_{i\bar{j}}%
^{T}(x,0)}(x_{0},y))}{\partial\nu}d\sigma(y)\\
& \geq & \int_{\Omega_{\alpha}}(\alpha-G_{g_{i\bar{j}}^{T}(x,0)}%
(x_{0},y))R^{T}(y,0)dv(y)\\
&  & -\int_{\partial\Omega_{\alpha}}F(y,t)\frac{\partial G_{g_{i\bar{j}}%
^{T}(x,0)}(x_{0},y))}{\partial\nu}d\sigma(y)\\
& \geq & -\int_{\Omega_{\alpha}}G_{g_{i\bar{j}}^{T}(x,0)}(x_{0},y)R^{T}%
(y,0)dv(y)\\
&  & -\int_{\partial\Omega_{\alpha}}F(y,t)\frac{\partial G_{g_{i\bar{j}}%
^{T}(x,0)}(x_{0},y))}{\partial\nu}d\sigma(y)
\end{array}
\label{14}%
\end{equation}
Set $r(\alpha)\geq1$ as following
\[%
\begin{array}
[c]{c}%
\alpha=\frac{r^{2}(\alpha)}{Vol(B_{0}(x_{0},r(\alpha)))}%
\end{array}
\]
Then by (\ref{14}) and the assumption (ii), we have
\[
C_{4}^{-1}r(\alpha)\leq r_{0}(x_{0},y)\leq C_{4}r(\alpha),
\]
for any $y\in\partial\Omega_{\alpha}$.
\begin{equation}%
\begin{array}
[c]{c}%
-\int_{\partial\Omega_{\alpha}}F(x,t)\frac{\partial G_{g_{i\bar{j}}^{T}%
(x,0)}(x_{0},y)}{\partial\nu}d\sigma(y)\geq\frac{C_{3}r(\alpha)}%
{Vol(B_{0}(x_{0},r(\alpha)))}\int_{\partial\Omega_{\alpha}}F(y,t)d\sigma(y)
\end{array}
\label{15}%
\end{equation}
and
\begin{equation}%
\begin{array}
[c]{ll}
& -\int_{\Omega_{\alpha}}G_{g_{i\bar{j}}^{T}(x,0)}(x_{0},y)R^{T}(y,0)dv(y)\\
\geq & -\frac{\sigma r^{2}(\alpha)}{\mathrm{Vol}(B_{0}(x_{0},r(\alpha)))}%
\int_{B_{0}(x_{0},C_{4}r(\alpha))}R^{T}(y,0)dv(y)\\
\geq & -\frac{\mathrm{Vol}(B_{\xi,g^{T}(x,0)}(x_{0},C_{4}r(\alpha
)))}{\mathrm{Vol}(B_{0}(x_{0},r(\alpha)))}\frac{\sigma r^{2}(\alpha
)}{\mathrm{Vol}(B_{\xi,g^{T}(x,0)}(x_{0},C_{4}r(\alpha)))}\int_{B_{\xi
,g^{T}(x,0)}(x_{0},C_{4}r(\alpha))}R^{T}(y,0)dv(y)\\
\geq & -\frac{\mathrm{Vol}(B_{\xi,g^{T}(x,0)}(x_{0},C_{4}r(\alpha
)))}{\mathrm{Vol}(B_{0}(x_{0},r(\alpha)))}\frac{C_{5}r^{2}(\alpha
)}{1+r^{1+\varepsilon}(\alpha)}\\
\geq & -C_{6}\frac{C_{5}r^{2}(\alpha)}{1+r^{1+\varepsilon}(\alpha)}\geq
-C_{7}r^{1-\varepsilon}(\alpha),
\end{array}
\label{16}%
\end{equation}
where $C_{6},C_{7}$ are positive constants depending only on $C_{1},C_{2}$ and
$n$. From (\ref{14}), (\ref{15}) and (\ref{16}), we have the following
\[%
\begin{array}
[c]{c}%
F(x_{0},t)\geq-C_{7}r^{1-\varepsilon}(\alpha)+\frac{C_{3}r(\alpha
)}{\mathrm{Vol}(B_{0}(x_{0},r(\alpha)))}\int_{\partial\Omega_{\alpha}%
}F(y,t)d\sigma(y)
\end{array}
\]
Integrating from $\frac{\alpha}{2}$ to $\alpha$, we have
\[%
\begin{array}
[c]{c}%
F(x_{0},t)\geq-C_{8}r^{1-\varepsilon}(\alpha)+\frac{C_{8}}{\mathrm{Vol}%
(B_{0}(x_{0},r(\alpha)))}\int_{B_{0}(x_{0},C_{4}r(\alpha))}F(y,t)dv(y)
\end{array}
\]
Since the Sasakian manifold $M^{2n+1}$ is a Riemannian manifold with respect
to $g_{i\bar{j}}(x,0)$, we can use the standard volume comparison for
Riemannian manifold and obtain the following
\[%
\begin{array}
[c]{c}%
F(x_{0},t)\geq-C_{9}a^{1-\varepsilon}+\frac{C_{9}}{\mathrm{Vol}(B_{0}%
(x_{0},a))}\int_{B_{0}(x_{0},a)}F(y,t)dv(y)
\end{array}
\]
for any positive constant $a$, where $C_{8},$ $C_{9}$ are positive constants
depending only on $\varepsilon,$ $C_{1},$ $C_{2}$ and $n$. Since the relation
between the transverse Ricci curvature and Ricci curvature, we have
\[
Ric\geq-2g(x,0).
\]
By using Theorem 1.4.2 in \cite{sy} and a simple scaling argument, there exist
a constant $C_{10}(n)>0$ depending only on $n$ such that for any fixed point
$x_{0}\in M^{2n+1}$ and a number $0<a<\infty$, there is a smooth function
$\psi(x)$ satisfying
\begin{equation}%
\begin{array}
[c]{c}%
e^{-C_{10}(n)(1+\frac{r_{0}(x_{0},x)}{a})}\leq\psi(x)\leq e^{-(1+\frac
{r_{0}(x_{0},x)}{a})}%
\end{array}
\label{19}%
\end{equation}
and
\begin{equation}%
\begin{array}
[c]{c}%
|\nabla\psi(x)|_{g(x,0)}\leq\frac{C_{10}}{a}\psi
\end{array}
\label{20}%
\end{equation}
and
\begin{equation}%
\begin{array}
[c]{c}%
|\Delta_{g(x,0)}\psi(x)|_{g(x,0)}\leq C_{10}\frac{1}{a^{2}}\psi(x).
\end{array}
\label{18}%
\end{equation}
It follows from (\ref{19}), (\ref{20}) and (\ref{18}) that%
\begin{equation}%
\begin{array}
[c]{ll}
& \frac{\partial}{\partial t}\int_{M}\psi(x)e^{F(x,t)}dv\\
\geq & \int_{M}\left(  \Delta_{B,g^{T}(x,0)}F(x,t)-R^{T}(x,0)\right)
\psi(x)dv\\
= & \int_{M}F(x,t)\Delta_{g(x,0)}\psi(x)dv-\int_{M}R^{T}(x,0)\psi(x)dv\\
\geq & \frac{C_{10}}{a^{2}}\int_{M}F(x,t)\psi(x)dv-\int_{M}R^{T}%
(x,0)\psi(x)dv\\
\geq & \frac{C_{10}}{a^{2}}\int_{M}F(x,t)\psi(x)dv-\frac{C_{11}}%
{a^{1+\varepsilon}}\mathrm{Vol}(B_{0}(x_{0},a))
\end{array}
\label{21}%
\end{equation}
Integrating (\ref{21}) from $0$ to $t$, we have%
\[%
\begin{array}
[c]{ll}
& \int_{M}\psi(x)e^{F(x,t)}dv-\int_{M}\psi(x)e^{F(x,0)}dv\\
\geq & \frac{C_{10}}{a^{2}}\int_{0}^{t}\int_{M}F(x,s)\psi(x)dvds-\frac
{C_{11}t}{a^{1+\varepsilon}}Vol(B_{0}(x_{0},a))\\
\geq & \frac{C_{10}t}{a^{2}}\int_{M}F(x,t)\psi(x)dv-\frac{C_{11}%
t}{a^{1+\varepsilon}}\mathrm{Vol}(B_{0}(x_{0},a)),
\end{array}
\]
that is,
\[%
\begin{array}
[c]{c}%
\int_{M}\psi(x)(1-e^{F(x,t)})dv\leq\frac{C_{10}t}{a^{2}}\int_{M}%
(-F(x,t))\psi(x)dv+\frac{C_{11}t}{a^{1+\varepsilon}}Vol(B_{0}(x_{0},a)),
\end{array}
\]
where we using the monotonicity of $F(x,t)$. In the following, we using the
argument in \cite{cz}, we have
\[%
\begin{array}
[c]{c}%
F(x,t)\geq-C(1+t)^{\frac{1-\varepsilon}{1+\varepsilon}},
\end{array}
\]
for $t\in\lbrack0,t_{\max})$. In particular, if $\varepsilon=1$, we have
\[
F(x,t)\geq-C.
\]
This completes the proof of this Lemma.
\end{proof}

The proof of \ the long-time solution as in Theorem \ref{T1A}:

\begin{proof}
Let $g_{i\bar{j}}^{T}$ be the maximal solution\ :of the Sasaki-Ricci flow
(\ref{9}). It follows from Lemma \ref{22} that
\[%
\begin{array}
[c]{c}%
\frac{\text{\textrm{det}}(g_{i\bar{j}}^{T}(x,t))}{\text{\textrm{det}}%
(g_{i\bar{j}}^{T}(x,0))}\geq e^{-c(1+t)^{\frac{1-\varepsilon}{1+\varepsilon}}}%
\end{array}
\]
on $M^{2n+1}\times\lbrack0,t_{\max})$, which implies
\begin{equation}%
\begin{array}
[c]{c}%
g_{i\bar{j}}^{T}(x,0)\geq g_{i\bar{j}}^{T}\geq e^{-c(1+t)^{\frac
{1-\varepsilon}{1+\varepsilon}}}g_{i\bar{j}}^{T}(x,0)
\end{array}
\label{24}%
\end{equation}
on $M^{2n+1}\times\lbrack0,t_{\max})$.

Since the Sasaki-Ricci flow equation (\ref{9}) is the parabolic version of the
Monge-Amp\`{e}re equation on the Sasakian manifold, the inequality (\ref{24})
is corresponding to the second order estimate for the Monge-Amp\`{e}re
equation. By adapting the standard argument of Calabi and Yau's arguments for
Monge-Amp\`{e}re equation and also Shi's arguments in \cite{s2}, we can show
that the derivative and higher order estimates for $g_{i\bar{j}}^{T}(x,t)$ are
uniformly bounded on any finite time interval, this implies that the solution
$g_{i\bar{j}}^{T}(x,t)$ exists for all $t\in\lbrack0,\infty)$. This complete
the proof of the existence of long time solution.
\end{proof}

(II) In the following, we will prove that the persevere property for the
Sasaki-Ricci flow which is to say that, for any $t>0,$ we have
\[
Rm^{T}(U,\overline{U},V,\overline{V})\geq0
\]
for any $U,$ $V\in D_{p}\subset T_{p}M$ and all $p\in M$ if $Rm^{T}\geq0$ for
$t=0.$

\begin{proof}
Now we compute the evolution of the transverse curvature $Rm^{T}$ along the
Sasaki-Ricci flow, we have
\begin{equation}%
\begin{array}
[c]{c}%
\frac{\partial Rm^{T}}{\partial t}=\Delta_{B}Rm^{T}+F(Rm^{T}),
\end{array}
\label{26}%
\end{equation}
where in the local transverse holomorphic coordinates, and
\begin{equation}%
\begin{array}
[c]{lll}%
F(Rm^{T})_{i\bar{\imath}j\bar{j}} & = & \sum_{p,q}Rm_{i\bar{\imath}p\bar{q}%
}^{T}Rm_{q\bar{p}j\bar{j}}^{T}-\sum_{p,q}|Rm_{i\bar{p}j\bar{q}}^{T}|^{2}%
+\sum_{p,q}|Rm_{i\bar{j}p\bar{q}}^{T}|^{2}\\
&  & -\sum_{p}Re(Rm_{i\bar{p}}^{T}R_{p\bar{\imath}j\bar{j}}^{T}+Rm_{j\bar{p}%
}^{T}Rm_{i\bar{\imath}p\bar{j}}^{T}).
\end{array}
\label{25}%
\end{equation}
For any tensor $S$ which has the same type as $Rm^{T}$, we can also define
$F(S)$ as in (\ref{25}). As in the K\"{a}hler case, by using the argument as
in Mok \cite{m} and Bando \cite{b}, $F(S)$ satisfies the null vector property.
That is, if there exist two nonzero vectors $U,V\in D^{1,0}$ such that
\[
S_{p}\geq0\text{ \textrm{and}\ }S_{p}(U,\overline{U};V,\overline{V})=0.
\]
Then
\[
F_{p}(S)(U,\overline{U};V,\overline{V})\geq0.
\]
By using this plus with \cite[Proposition 1.]{b} where the Hamilton maximum
principle for tensors was using, we can show that the transverse bisectional
curvature is nonnegative for all $t$. The only thing difference is that the
auxiliary function should be a basic function in our Sasaki case. We refer to
\cite[Proposition 8.4.]{he} for some details.
\end{proof}

(III) In order to estimate: For any integer $m\geq0$, there is a constant $C$
such that%
\[%
\begin{array}
[c]{c}%
||\nabla^{m}Rm^{T}||^{2}(x,t)\leq\frac{C}{t^{1+m}},\text{\ \textrm{for all}
}t>0,x\in M.
\end{array}
\]
We rewrite the equation (\ref{26}) as following
\[%
\begin{array}
[c]{c}%
\frac{\partial}{\partial t}Rm^{T}=\Delta Rm^{T}+Rm^{T}\ast Rm^{T},
\end{array}
\]
where we use the fact $\nabla_{\xi}Rm^{T}=0$, and so we see the norm $|\nabla
Rm^{T}|$ agrees with the norm when we replace $g$ by $g^{T}$. With this mind,
we have
\[%
\begin{array}
[c]{c}%
\frac{\partial}{\partial t}|\nabla Rm^{T}|^{2}=\Delta|\nabla Rm^{T}%
|^{2}-2|\nabla^{2}Rm^{T}|^{2}Rm^{T}\ast(\nabla Rm^{T})^{\ast2}+(\nabla
Rm^{T})^{\ast2}%
\end{array}
\]
and%
\[%
\begin{array}
[c]{ccl}%
\frac{\partial}{\partial t}\left(  t|\nabla Rm^{T}|^{2}+\beta|Rm^{T}%
|^{2}\right)  & \leq & \Delta\left(  t|\nabla Rm^{T}|^{2}+\beta|Rm^{T}%
|^{2}\right) \\
&  & +(1+C_{3}t|Rm^{T}|+C_{4}t-2\beta)|\nabla Rm^{T}|^{2}+C_{5}|Rm^{T}|^{3}.
\end{array}
\]
Suppose that $|Rm^{T}|\leq K$. Set $2\beta\geq1+C_{3}t|Rm^{T}|+C_{4}t$, then
we have
\[%
\begin{array}
[c]{c}%
\frac{\partial}{\partial t}\left(  t|\nabla Rm^{T}|^{2}+\beta|Rm^{T}%
|^{2}\right)  -\Delta\left(  t|\nabla Rm^{T}|^{2}+\beta|Rm^{T}|^{2}\right)
\leq C_{5}K^{3}.
\end{array}
\]
By applying the maximum principle, we have
\[
\left(  t|\nabla Rm^{T}|^{2}+\beta|Rm^{T}|^{2}\right)  \leq\beta K^{2}%
+C_{5}\beta K^{3}t.
\]
Using the choice of $\beta,$ there exists a constant $C_{1}$ such that for any
$t\in(0,\frac{\alpha}{K})$%
\[
|\nabla Rm^{T}|\leq C_{1}t^{-\frac{1}{2}}\max\{\sqrt{K},K\}
\]
for $m=1.$ By induction on $m$, we have
\[%
\begin{array}
[c]{c}%
||\nabla^{m}Rm^{T}||(x,t)\leq\frac{C_{m}\max\{\sqrt{K},K\}}{t^{\frac{m}{2}}}%
\end{array}
\]
for any integer $m\geq0$.

(IV) From Lemma \ref{22}, we can obtain the result as follows.

\begin{corollary}
\label{C31}Let $(M^{2n+1},J,\theta)$ be a noncompact Sasakian $(2n+1)$%
-manifold with bounded and nonnegative transverse bisectional curvature and
$n\geq2$. Suppose for a fixed base point $x_{0}$,%
\[%
\begin{array}
[c]{c}%
\mathrm{Vol}(B_{\xi}(x_{0},r))\geq C_{7}r^{2n},\\
\frac{1}{\mathrm{Vol}(B_{\xi}(x,r))}\int_{B_{\xi}(x,r)}R^{T}(y)dy\leq
\frac{C_{8}}{1+r_{0}^{2}(x_{0},x)}.
\end{array}
\]
Let $g_{i\bar{j}}^{T}$ be the solution of the Sasaki-Ricci flow (\ref{9}) with
$g^{T}(x)$ as initial transverse K\"{a}hler metric, and let $\mathrm{Vol}%
_{t}(B_{\xi}(x_{0},r))$ be the volume of the transverse ball of radius $r$ and
centered at $x_{0}$ with respect to $g^{T}(x,t)$. Then there exists a positive
constant $\kappa$ such that
\[
\mathrm{Vol}_{t}(B_{\xi}(x_{0},r))\geq\kappa r^{2n}%
\]
for all $t\in\lbrack0,t_{\max})$ and $0\leq r\leq\infty.$
\end{corollary}

\begin{proof}
We have known that the transverse Ricci curvature of the solution $g_{i\bar
{j}}^{T}(x,t)$ is nonnegative, so that the metric is shrinking under that flow
(\ref{9}). Since the metric $g(x,t)$ is given by
\begin{equation}
g(x,t)=g^{T}(x,t)+\widetilde{\eta}\otimes\widetilde{\eta} \label{36}%
\end{equation}
with $\widetilde{\eta}=\eta+d_{B}^{c}\varphi$ and $\widetilde{\eta}(\xi)=1$
for the fixed $\xi$, we have
\begin{equation}%
\begin{array}
[c]{lll}%
\mathrm{Vol}_{t}(B_{\xi,t}(x_{0},r)) & \geq & \mathrm{Vol}_{t}(B_{\xi,0}%
(x_{0},r))\\
& = & \int_{B_{\xi,0}(x_{0},r)}e^{F(x,t)}dV_{0}\\
& = & \mathrm{Vol}(B_{\xi,0}(x_{0},r))+\int_{B_{\xi,0}(x_{0},r)}\left(
e^{F(x,t)}-1\right)  dV_{0}%
\end{array}
\label{29}%
\end{equation}
It follows from Lemma \ref{22} that
\begin{equation}%
\begin{array}
[c]{lll}%
\int_{B_{\xi,0}(x_{0},r)}\left(  e^{F(x,t)}-1\right)  dV_{0} & \geq &
C_{13}\int_{M}\left(  e^{F(x,t)}-1\right)  \psi(x)dV_{0}\\
& \geq & C_{14}t[\frac{F_{min}(t)}{r^{2}}-\frac{1}{r^{2}}]\mathrm{Vol}%
(B_{\xi,0}(x_{0},r))
\end{array}
\label{30}%
\end{equation}
for some positive constants $C_{13},C_{14}$ depending on $C_{1},C_{12}$ and
$n$. By (\ref{29}) and (\ref{30}),
\[%
\begin{array}
[c]{c}%
\lim_{r\rightarrow\infty}\frac{\mathrm{Vol}_{t}(B_{\xi,t}(x_{0},r))}{r^{2n}%
}\geq\lim_{r\rightarrow\infty}\frac{\mathrm{Vol}_{t}(B_{\xi,0}(x_{0}%
,r))}{r^{2n}}\geq C_{1}.
\end{array}
\]
This completes this Corollary.
\end{proof}

Now we are ready to prove $(3)$ of Theorem \ref{T1A}.

By using trace of Li-Yau-Hamilton matrix type inequality (\ref{09}) for the
transverse scalar curvature%
\[%
\begin{array}
[c]{c}%
\frac{\partial R^{T}}{\partial t}-\frac{|\nabla^{T}R^{T}|^{2}}{R^{T}}%
+\frac{R^{T}}{t}\geq0,
\end{array}
\]
it follows that $tR^{T}(\cdot,t)$ is nondecreasing in time. It then follows
from (\ref{8}) that for $x\in M$ and $t\in\lbrack0,+\infty)$,
\begin{equation}%
\begin{array}
[c]{lll}%
-F(x,2t) & = & \int_{t}^{2t}R^{T}(x,s)ds\geq\int_{t}^{2t}R^{T}(x,s)ds\\
& \geq & tR^{T}\int_{t}^{2t}\frac{1}{s}ds=\left(  \log2\right)  \cdot
tR^{T}(x,t).
\end{array}
\label{31}%
\end{equation}
By using the formula (\ref{32}), we have
\begin{equation}%
\begin{array}
[c]{l}%
R^{T}(x,t)\leq\frac{C_{15}}{(1+t)}%
\end{array}
\label{33}%
\end{equation}
for some positive constant $C_{15}$ depending only on $\varepsilon,C_{1}%
,C_{2}$ and $n$. By Corollary \ref{C31}, we have known that the maximal volume
growth condition is preserved under the Sasaki-Ricci flow (\ref{9}). That is
\[
\mathrm{Vol}_{t}(B_{\xi,t}(x_{0},r))\geq C_{1}r^{2n}%
\]
for all $t\in\lbrack0,t_{\max})$ and $0<r\leq\infty.$ Since (\ref{36}), it
follows that
\[
\mathrm{Vol}(B_{t}(x_{0},r))\geq Cr^{2n+1}.
\]
By the local injective radius estimate of Cheeger-Gromov-Taylor (\cite{cgt}),
we can obtain that
\begin{equation}%
\begin{array}
[c]{l}%
\mathrm{inj}_{(M,g_{i\bar{j}})}(p)\geq C>0.
\end{array}
\label{35}%
\end{equation}
By the standard scaling argument, it follows from (\ref{33}) and (\ref{35}),
we have the following
\begin{equation}%
\begin{array}
[c]{l}%
\mathrm{inj}(M,g_{i\overline{j}}(\cdot,t))\geq C_{2}t^{\frac{1}{2}}%
\end{array}
\label{37}%
\end{equation}
for $t\geq1$.

Finally, we will give the Proof of Theorem \ref{T2B}\textbf{: }

\begin{proof}
By using assumption and Proposition \ref{P1A}, there exists a nonconstant
basic CR holomorphic function $f$ of polynomial growth of degree at most $d$.
Set
\[
u(x)=\log\left(  1+|f(x)|^{2}\right)  .
\]
Then we have the basic CR plurisubharmonic function which satisfies the
following:
\[%
\begin{array}
[c]{c}%
\sqrt{-1}\partial_{B}\overline{\partial}_{B}u\geq0
\end{array}
\]
and
\[
u(x)\leq C\left(  2+\log r(x)\right)  .
\]
Now using $u(x)$ as a weighted function, by the $L^{2}$-estimate
(\cite[Proposition 2.1.]{chll}), we have a nontrivial basic section $S^{T}%
\in\emph{O}_{d}(M,K_{M}^{T})$ for some $d$. Then by the CR Poincare-Lelong
formula
\[%
\begin{array}
[c]{c}%
\sqrt{-1}\partial_{B}\overline{\partial}_{B}||s^{T}||_{h}^{2}=-c_{1}^{B}%
(L^{T},h)+[Z_{s^{T}}]
\end{array}
\]
for any basic CR-holomorphic line bundle $(L^{T},h)$. In particular for
$\rho^{T}=-\sqrt{-1}R_{ij}^{T}dz^{i}d\overline{z}^{j}$, we have
\[%
\begin{array}
[c]{c}%
\sqrt{-1}\partial_{B}\overline{\partial}_{B}\log||s^{T}||^{2}=-\rho^{T}+\cdots
\end{array}
\]
Furthermore, we have
\begin{equation}%
\begin{array}
[c]{c}%
\Delta\log||s^{T}||^{2}\geq g^{i\overline{j}}R_{i\overline{j}}^{T}=R^{T}\geq0.
\end{array}
\label{2}%
\end{equation}
Now by Lemma 4.1 and Lemma 4.2 of \cite{chl1}, we can solve the CR heat
equation
\[%
\begin{array}
[c]{c}%
\begin{cases}
(\frac{\partial}{\partial t}-\Delta)v(x,t)=0\\
v(x,0)=2\log||s^{T}||^{2}%
\end{cases}
\end{array}
\]
and we have
\[%
\begin{array}
[c]{c}%
\lim{}_{t\rightarrow\infty}\sup\frac{v(x,t)}{\log t}\leq d.
\end{array}
\]
Moreover, for $w(x,t)=\frac{\partial}{\partial t}v(x,t)$ which is
\begin{equation}%
\begin{array}
[c]{c}%
w(x,t)=\int_{M}H(x,y,t)\left(  \Delta\log||s^{T}||(y)\right)  dy,
\end{array}
\label{3}%
\end{equation}
so that we have
\[%
\begin{array}
[c]{c}%
\begin{cases}
(\frac{\partial}{\partial t}-\Delta)w(x,t)=0\\
w(x,0)=2\log||s^{T}||,
\end{cases}
\end{array}
\]
where $s^{T}\in\emph{O}_{d}^{CR}(M,K_{M}^{T})$. Finally, by (\ref{1}) and the
proof as in \cite[Theorem 4.1.]{chl1} and the moment-type estimate of
\cite{cf}, we come out with
\[%
\begin{array}
[c]{c}%
\frac{\partial}{\partial t}\left(  tw(x,t)\right)  \geq0
\end{array}
\]
and
\begin{equation}%
\begin{array}
[c]{c}%
\lim_{t\rightarrow\infty}tw(x,t)\leq Cd.
\end{array}
\label{4}%
\end{equation}
On the other hand, it follows from (\ref{2}) and (\ref{3}) that
\[%
\begin{array}
[c]{c}%
w(x,t)\geq\int_{M}H(x,y,t)R^{T}(y)dy.
\end{array}
\]
By using (4.6) in \cite{chl1} and the CR heat kernel estimate as Proposition
3.1 in \cite{cchl}, we have
\begin{equation}%
\begin{array}
[c]{c}%
\frac{C}{\mathrm{Vol}(B_{\xi}(x,\sqrt{t}))}\int_{B_{\xi}(x,\sqrt{t})}^{T}%
R^{T}(y)dy\leq w(x,t).
\end{array}
\label{5}%
\end{equation}
By using (\ref{4}) and (\ref{5}), we have
\[%
\begin{array}
[c]{c}%
\frac{1}{\mathrm{Vol}(B_{\xi}(x,r))}\int_{B_{\xi}(x,r)}R^{T}(y)dy\leq\frac
{C}{1+r^{2}}.
\end{array}
\]
This completes the proof of the Theorem.
\end{proof}

\section{CR Yau Uniformization Conjecture}

In the following, we prove the main result of the present paper.

The Proof of Theorem \ref{T2A}: (1) It follows from the injectivity radius
estimate (\ref{37}) that the exponential map gives a diffeomorphism between
the balls of $M$ and the Euclidean space as in (2) below. The key is to modify
the exponential maps to become CR biholomorphims. By using the foliation chart
system on $M$, on each $V_{\alpha}$, for a fixed $p_{0}\in M$ and a fixed $t$,
we have
\[%
\begin{array}
[c]{c}%
d\pi_{\alpha}:D_{p_{0}}\rightarrow T_{\pi_{\alpha}(p_{0})}V_{\alpha}%
\cong\mathbb{C}^{n}.
\end{array}
\]
Since $\Phi(Z_{i})=\sqrt{-1}Z_{i}$ and $\Phi(\overline{Z}_{i})=-\sqrt
{-1}\overline{Z}_{i}$, we can define $e_{i}=Z_{i}+\overline{Z}_{i}$ and then
$\Phi(e_{i})=\sqrt{-1}(Z_{i}-\overline{Z}_{i})$ such that $\left\{  e_{1}%
,\Phi(e_{i}),\cdots,e_{n},\Phi(e_{n})\right\}  $ form an orthonormal frame on
$D_{p_{0}}$. For any $v\in D_{p_{0}}$, we have
\[
v=x_{1}e_{1}+y_{1}\Phi(e_{1})+\cdots+x_{n}e_{n}+y_{n}\Phi(e_{n}).
\]
Now we can define the real linear isomorphism $L:D_{p_{0}}\rightarrow
\mathbb{C}^{n}$ by
\[
L(v)=\left(  z_{1},z_{2},\cdots,z_{n}\right)  \in\mathbb{C}^{n},
\]
where $z_{i}=x_{i}+\sqrt{-1}y_{i}$, $i=1,2,\cdots,n$. As in \cite{co}, we
define the geodesic tube ball around the $P=\overline{\text{orb}_{\xi}p_{0}}$
of the radius $r$ is the set
\[%
\begin{array}
[c]{c}%
T(P,r)=\cup_{p\in P}\{\text{exp}_{p}(v)|v\in T_{p}P^{\bot},||v||<r\}
\end{array}
\]
which is the same as
\[
T(P,r)=B_{\xi}(p_{0},r),
\]
if $r<\mathrm{inj}(M)$. Now we use $\text{exp}_{p_{0}}^{T}$ to denote the
transverse exponential map with respect to the transverse K\"{a}hler metric
$g_{i\overline{j}}^{T}$. By the injectivity estimate, we can know that the
map
\[%
\begin{array}
[c]{c}%
\varphi_{t}=\exp_{p_{0}}^{T}\circ L^{-1}:\widehat{B}(0,C(t+1)^{\frac{1}{2}%
})\subset\mathbb{C}^{n}\rightarrow B_{\xi}(p_{0},C(t+1)^{\frac{1}{2}})\subset
M
\end{array}
\]
is a diffeomorphism, where $\widehat{B}(0,C(1+t)^{\frac{1}{2}})$ is the
standard ball in $\mathbb{C}^{n}$ and $B_{\xi}(p_{0},C(1+t)^{\frac{1}{2}})$ is
the transverse ball of $M$ with respect to the transverse K\"{a}hler metric
$g_{i\overline{j}}^{T}$. We consider the pull back metric in real coordinates
$g_{ij}^{\ast T}(x,t)$ and complex coordinates $g_{AB}^{\ast T}(x,t)$,
\[%
\begin{array}
[c]{c}%
\varphi_{t}^{\ast}(g_{ij}^{T}(x,t))=g_{ij}^{\ast T}(x,t)dx_{i}dx_{j}%
=g_{AB}^{\ast T}(x,t)dz^{A}dz^{B}%
\end{array}
\]
on $\widehat{B}(0,C(1+t)^{\frac{1}{2}})$, where $A,B=\alpha$ or $\overline
{\alpha}$ $(\alpha=1,2,\cdots,n)$. Since $\varphi_{t}$ is not CR-holomorphic
in general, the transverse metric $g_{AB}^{\ast T}(\cdot,t)$ is not Hermitian
for the standard complex structure in $\widehat{B}(0,C(1+t)^{\frac{1}{2}%
})\subset\mathbb{C}^{n}$. We need the following Lemma which is due to Hamilton
\cite[Theorem 4.10.]{h2}.

\begin{lemma}
Suppose the metric $g_{ij}dx^{i}dx^{j}$ is in geodesic coordinates. Suppose
the Riemannian curvature $Rm$ is bounded between $-B_{0}$ and $B_{0}$. Then
there exist positive constants $c,C_{0}$ depending only on the dimension such
that for any $|x|\leq\frac{c}{\sqrt{B_{0}}}$, the following holds
\[%
\begin{array}
[c]{c}%
|g_{ij}-\delta_{ij}|\leq C_{0}B_{0}|x|^{2}.
\end{array}
\]
Furthermore, if in addition $|\nabla Rm|\leq B_{0}$ and $|\nabla^{2}Rm|\leq
B_{0}$, then
\[%
\begin{array}
[c]{c}%
|\frac{\partial}{\partial x^{j}}g_{kl}|\leq C_{0}B_{0}|x|,\text{ \textrm{and}%
}\ |\frac{\partial^{2}}{\partial x^{i}\partial x^{j}}g_{kl}|\leq C_{0}B_{0}%
\end{array}
\]
for any $|x|\leq\frac{x}{\sqrt{B_{0}}}$.
\end{lemma}

By the assumption and the above lemma, we can obtain some positive constants
$\overline{c},C$, and get the following inequalities for the evolving
transverse metric $g_{\alpha\overline{\beta}}^{\ast T}$,
\begin{align*}
|g_{\alpha\overline{\beta}}^{\ast T}(z,t)-\delta_{\alpha\beta}|_{t}  &  \leq
C|z|^{2}(1+t)^{-1}\\
|g_{\overline{\alpha}\beta}^{\ast T}(z,t)-\delta_{\alpha\beta}|_{t}  &  \leq
C|z|^{2}(1+t)^{-1}\\
|g_{\alpha\beta}^{\ast T}(z,t)|_{t}  &  \leq C|z|^{2}(1+t)^{-1}\\
|g_{\overline{\alpha}\overline{\beta}}^{\ast T}(z,t)|_{t}  &  \leq
C|z|^{2}(1+t)^{-1}\\
|\widehat{\nabla}g_{AB}^{\ast T}(z,t)|_{t}  &  \leq C|z|^{2}(1+t)^{-1}\\
|\widehat{\nabla}\widehat{\nabla}g_{AB}^{\ast T}(z,t)|_{t}  &  \leq
C(1+t)^{-1}%
\end{align*}
where $z\in\widehat{B}(0,C(1+t)^{\frac{1}{2}})$, and $t\in\lbrack1,\infty)$,
and $|\cdot|_{t}$ is the norm with respect to the transverse metric
$g_{ij}^{\ast T}(z,t)$. Let $\varphi_{t}^{\ast}\Phi$ and $\overline{\partial
}_{t}=\varphi_{t}^{\ast}(\overline{\partial}_{B})$ be the pull back complex
structure and $\overline{\partial}_{B}$-operator of $M$ on $B(0,C(1+t)^{\frac
{1}{2}})$. We can know that $\varphi_{t}$ is CR-holomorphic with respect to
$\varphi_{t}^{\ast}\Phi$. But the functions $z^{\alpha}$ $(\alpha
=1,2,\cdots,n)$ are not holomorphic with respect to $\varphi_{t}^{\ast}\Phi$.
This is just indicates the difference of $\varphi_{t}^{\ast}\Phi$ with the
standard complex structure $J_{\mathbb{C}^{n}}$ on $\hat{B}(0,C(1+t)^{\frac
{1}{2}})$. Now we can estimate the difference between them like the method in
\cite{cz}. We denote $\{x^{1},\cdots,x^{n},x^{n+1},\cdots,x^{2n}%
\}=\{x^{1},\cdots,x^{n},y^{1},\cdots,y^{n}\}$ as real coordinates for $\hat
{B}(0,C(1+t)^{\frac{1}{2}})$. Set
\[%
\begin{array}
[c]{c}%
\varphi_{t}^{\ast}\Phi=\sum_{i,j=1}^{2n}\Phi_{j}^{i}\frac{\partial}{\partial
x_{i}}\otimes dx^{j}%
\end{array}
\]
and
\[%
\begin{array}
[c]{c}%
J_{C^{n}}=\sum_{i.j=1}^{2n}J_{j}^{i}\frac{\partial}{\partial x_{i}}\otimes
dx^{j},
\end{array}
\]
where $\Phi_{j}^{i}$ is the representation of $\Phi$ in the normal coordinate
at $x_{0}$, and $\Phi_{j}^{i}(0,t)=J_{j}^{i}(0)$.\ Let $\nabla^{t}$ and
$\Gamma_{ij}^{tk}$ be the covariant derivative and Christoffel symbols with
respect to the pull back metric $g_{ij}^{\ast T}$ of $g_{\alpha\beta}^{\ast
T}$ on $B(0,C(1+t)^{\frac{1}{2}})$. Set
\[%
\begin{array}
[c]{c}%
H_{k}^{j}=J_{j}^{i}-\Phi_{k}^{j}.
\end{array}
\]
By computation, we have
\[%
\begin{array}
[c]{c}%
\sum_{i}x^{i}\nabla_{\frac{\partial}{\partial x_{i}}}^{t}[\sum_{j,k}H_{k}%
^{j}\frac{\partial}{\partial x_{j}}\otimes dx^{k}]=x^{i}\Gamma_{ip}^{tj}%
J_{k}^{p}\frac{\partial}{\partial x_{j}}\otimes dx^{k}-x^{i}\Gamma_{ik}%
^{tp}J_{p}^{j}\frac{\partial}{\partial x_{j}}\otimes dx^{k}.
\end{array}
\]
That is
\[%
\begin{array}
[c]{c}%
x^{i}\nabla_{\frac{\partial}{\partial x^{i}}}H_{k}^{j}=x^{i}\Gamma_{ip}%
^{tj}J_{k}^{p}-x^{i}\Gamma_{ik}^{tp}J_{p}^{j}.
\end{array}
\]
From the Gauss Lemma, we
\[
g_{ij}^{\ast T}x^{i}=\delta_{ij}x^{i}%
\]
on $\hat{B}(0,C(1+t)^{\frac{1}{2}})$. As \cite{h2} and \cite{cz}, we define
the symmetric tensor
\[%
\begin{array}
[c]{c}%
A_{ij}=\frac{1}{2}x^{k}\frac{\partial}{\partial x^{k}}g_{ij}^{\ast T}.
\end{array}
\]
Then we can have
\[%
\begin{array}
[c]{c}%
x^{j}\Gamma_{jk}^{ti}=g^{\ast Til}A_{kl}.
\end{array}
\]
So we can obtain the following
\[%
\begin{array}
[c]{c}%
|x^{i}\nabla_{\frac{\partial}{\partial x_{i}}}^{t}H_{k}^{j}|_{t}\leq
C|A_{ij}|_{t}%
\end{array}
\]
on $\hat{B}(0,C(1+t)^{\frac{1}{2}})$, where $C$ is a positive constant
depending only on $n$.

If the transversal metric $g_{ij}^{T}$ has $|Rm^{T}|\leq B_{0}$ in the
geodesic ball of radius $r\leq\frac{c}{\sqrt{B_{0}}}$, there exists $C<\infty$
such that
\[
|A_{ij}|\leq CB_{0}r^{2}.
\]
Then
\begin{equation}%
\begin{array}
[c]{c}%
|x^{i}\nabla_{\frac{\partial}{\partial x_{i}}}H_{k}^{j}|\leq C|x|^{2}%
(1+t)^{-\frac{1}{2}}%
\end{array}
\label{66}%
\end{equation}
on $\hat{B}(0,C(1+t)^{\frac{1}{2}})$ for some positive constant $C$ depending
only on $n$. Set
\[%
\begin{array}
[c]{c}%
M(r)=\sup_{\{|x|\leq r\}}|H_{k}^{j}|_{t}.
\end{array}
\]
By using (\ref{66}), we can obtain that
\[
M(r)\leq Cr^{2}(1+t)^{-1}%
\]
on $\hat{B}(0,C(1+t)^{\frac{1}{2}})$. Since $z^{\alpha}$ $(\alpha
=1,2,\cdots,n)$ are holomorphic with respect to $J_{\mathbb{C}^{n}}$, we can
obtain that%
\begin{equation}%
\begin{array}
[c]{lll}%
|\overline{\partial}_{B}^{t}z^{\alpha}|_{t} & \leq & \sum_{i}|\frac{\partial
}{\partial x^{i}}z^{\alpha}+\sqrt{-1}\varphi_{t}^{\ast}\Phi(\frac{\partial
}{\partial x^{i}})z^{\alpha}|\\
& \leq & \sum_{i}|\frac{\partial}{\partial x^{i}}z^{\alpha}+\sqrt{-1}J_{C^{n}%
}(\frac{\partial}{\partial x^{i}})z^{\alpha}|+\sum_{i}|(\varphi_{t}^{\ast}%
\Phi-J_{C^{n}})(\frac{\partial}{\partial x^{i}})z^{\alpha}|\\
& = & \sum_{i}|(\varphi_{t}^{\ast}\Phi-J_{C^{n}})(\frac{\partial}{\partial
x^{i}})z^{\alpha}|\\
& \leq & C|z|^{2}(1+t)^{-1},
\end{array}
\label{7}%
\end{equation}
where $C$ is a positive constant depending only on $n$. For any fixed $t\geq
1$, we consider the $\bar{\partial}$-equation
\begin{equation}
\bar{\partial}\zeta^{\alpha}=\bar{\partial}z^{\alpha}\label{88}%
\end{equation}
on $\hat{B}(0,r(t))$, where $r(t)=(\frac{C}{2})^{\frac{1}{2}}(1+t)^{\frac
{1}{4}}$. By (\ref{7}), we obtain
\[%
\begin{array}
[c]{c}%
|\bar{\partial}^{t}z^{\alpha}|_{t}\leq\widetilde{C}\frac{C}{2}(1+t)^{-\frac
{1}{2}}%
\end{array}
\]
on $\hat{B}(0,r(t))$. By using $L^{2}$ estimate theory for $\bar{\partial}%
$-operator, we know that the equation (\ref{88}) has smooth solutions
$\{\zeta^{\alpha}(z,t)|$ $\alpha=1,2,\cdots,n\}$ with the following
properties
\[%
\begin{array}
[c]{c}%
|\zeta^{\alpha}(z,t)|\leq\frac{C}{r(t)}%
\end{array}
\]
and
\[%
\begin{array}
[c]{c}%
|\widehat{\nabla}\zeta^{\alpha}|\leq\frac{C}{r^{2}(t)}%
\end{array}
\]
on $\hat{B}(0,r(t))$. From the equation (\ref{88}), we can have a holomorphic
\[%
\begin{array}
[c]{c}%
\Psi_{t}=(\Psi_{t}^{1},\cdots,\Psi_{t}^{n}):(\widehat{B}(0,r(t)),\varphi
_{t}^{\ast}\Phi)\rightarrow(\mathbb{C}^{n},J_{\mathrm{can}})
\end{array}
\]
with
\[
\Psi_{i}^{\alpha}=z^{\alpha}-\zeta^{\alpha}(x,t).
\]
When $t\rightarrow\infty,$ $\Psi_{t}$ is a diffeomorphism and
\[
\Psi_{t}\circ\varphi_{t}^{-1}:B_{\xi}(p_{0},r(t))\rightarrow\mathbb{C}^{n}%
\]
is a holomorphic and injective map and the image of $B_{\xi}(p_{0},r(t))$
contains the Euclidean ball $\widehat{B}(0,r(t))\subset\mathbb{C}^{n}.$
Finally, we denote
\[%
\begin{array}
[c]{c}%
\Omega_{t}=(\Psi_{t}\circ\varphi_{t}^{-1})^{-1}(\widehat{B}(0,\frac{1}%
{2}r(t))).
\end{array}
\]
Then for any $t,$
\[%
\begin{array}
[c]{c}%
B_{\xi}^{0}(p_{0},\frac{1}{4}r(t))\subset B_{\xi}^{t}(p_{0},\frac{1}%
{4}r(t))\subset\Omega_{t}.
\end{array}
\]
Hence there exists a sequence of $t_{k}\rightarrow\infty$ such that
$\Omega_{t_{1}}\subset\Omega_{t_{2}}\subset\cdots.$ Then by gluing argument of
Shi \cite{s2}, we have the bi-holomorphic map from $\cup\Omega_{t_{k}}$ to a
pseudoconvex domain $\Omega$ of $\mathbb{C}^{n}.$

We observe that $\{Z_{\alpha}=\frac{\partial}{\partial z^{\alpha}}+\sqrt
{-1}\overline{z}^{\alpha}\frac{\partial}{\partial s}\}_{\alpha=1}^{n}$ is
exactly a local frame in the $(2n+1)$-dimensional Heisenberg group
$\mathbb{H}_{n}=\mathbb{C}^{n}\times\mathbb{R}$ as the standard contact
Euclidean space $(\mathbb{R}^{2n+1},\eta_{\mathrm{can}},\Phi,\xi
,g_{\mathrm{can}})$ with coordinates $(x_{1},...,x_{n},y_{1},...,y_{n},s)$.
Take is a pseudohermitian contact structure on $\mathbb{H}_{n}$ as the contact
$1$-form
\begin{equation}
\eta_{\mathrm{can}}=\frac{1}{2}ds-\frac{1}{4}\sum_{i=1}^{n}(y_{i}dx_{i}%
-x_{i}dy_{i})=\frac{1}{2}(ds+\sqrt{-1}\sum\left(  z^{\alpha}d\overline
{z}^{\alpha}-\overline{z}^{\alpha}dz^{\alpha}\right)  )\label{2026A}%
\end{equation}
with the Reeb vector field $\xi=2\frac{\partial}{\partial s}$, the associated
Sasaki metric
\begin{equation}
g_{\mathrm{can}}=\frac{1}{4}\sum_{i=1}^{n}(dx_{i}^{2}+dy_{i}^{2}%
)+\eta_{\mathrm{can}}\otimes\eta_{\mathrm{can}}\label{2026B}%
\end{equation}
and
\[
\Phi=\sum_{i=1}^{n}(-dx_{i}\otimes\frac{\partial}{\partial y_{i}}%
+dy_{i}\otimes\frac{\partial}{\partial x_{i}}+\frac{x_{i}}{2}dx_{i}%
\otimes\frac{\partial}{\partial z}+\frac{y_{i}}{2}dy_{i}\otimes\frac{\partial
}{\partial z})
\]
with an orthonormal frame $\{E_{i}\}_{i=1}^{2n+1}$:
\begin{align*}
&  E_{i}=2\frac{\partial}{\partial x_{i}}+y_{i}\frac{\partial}{\partial s};\\
&  E_{n+i}=-2\frac{\partial}{\partial y_{i}}+x_{i}\frac{\partial}{\partial
s};\\
&  E_{2n+1}=\xi=2\frac{\partial}{\partial s}%
\end{align*}
such that $\Phi(E_{i})=E_{n+i}$ and $\Phi(\xi)=0$.

Therefore, by defining along the corresponding geodesic integral curve with
respect to the Killing Reeb vector field $\xi=2\frac{\partial}{\partial s}$
along the slice $D_{\alpha}$, it follows that we have the CR-biholomorphic
from $M$ \ to $\Omega\times\mathbb{R}\subset\mathbb{C}^{n}\times\mathbb{R}.$
For the explicite construction, we refer to the proof  of Theorem 4.1 as in
the paper of Chang-Han-Li-Lin \cite{chll}.

(2) From Theorem \ref{T1A}, the transverse bisectional curvature of
$g_{i\overline{j}}^{T}(\cdot,t)$ is nonnegative for all $x\in M$ and $t\geq0$.
From the Sasaki-Ricci flow (\ref{9}), (\ref{8}) and the equation (\ref{36}),
we know that the ball $B_{t}(x_{0},C_{2}t^{\frac{1}{2}})$ of radius
$C_{2}t^{\frac{1}{2}}$ with respect the metric $g(\cdot,t)=g^{T}%
(\cdot,t)+\widetilde{\eta}\otimes\widetilde{\eta}$, contains the ball
$B_{0}(x_{0},C_{2}t^{\frac{1}{2}})$ with respect to the metric $g(\cdot
,0)=g^{T}(\cdot,0)+\eta\otimes\eta$. By using (\ref{37}), we can obtain that
\begin{equation}%
\begin{array}
[c]{c}%
\pi_{p}(M)=0\quad\text{and}\quad\pi_{q}(M,\infty)=0
\end{array}
\end{equation}
for any $p\geq1$ and $1\leq q\leq2n-1$, where $\pi_{q}(M,\infty)$ is the
$q$-th homotopy group of $M$ at infinity. By generalised Poincar\'{e}
conjecture (\cite{f}, \cite{s}), we obtain that $M^{2n+1}$ is homeomorphic to
$\mathbb{R}^{2n+1}$. Moreover, due to Gompf's result (\cite{go}),
$\mathbb{R}^{4}$ has exotic differential structures and the homeomorphisms can
be diffeomorphisms for $n>2$.

\appendix

\section{  \ \ \ }

In this appendix, we will devote to have the sketch proof of Proposition
\ref{P1} by following the method as in the paper of Lee-Tam \cite{lt2}.

As in \cite{swz}, we have

\begin{proposition}
\label{AP1}Let $M$ be a compact Riemannian $(2n+1)$-manifold. For any Sasakian
structure $(M,\xi,\eta,g^{T},\Phi,\omega),$ there exists a family of Sasakian
structures $(M,\xi(t),\eta(t),g^{T}(t),\Phi(t),\omega(t)),t\in\lbrack
0,T_{\max})$ for some $T_{\max}>0$ satisfying the Sasaki-Ricci flow%
\begin{equation}%
\begin{array}
[c]{lll}%
\frac{d}{dt}g_{i\overline{j}}^{T}(x,t) & = & -R_{i\overline{j}}^{T}(x,t),\\
g_{i\overline{j}}^{T}(x,0) & = & g_{i\overline{j}}^{T}(x)
\end{array}
\label{2020B}%
\end{equation}
with the initial condition $(\eta(0),\Phi(0),g(0))=(\eta,\Phi,g).$
\end{proposition}

In fact, the Sasaki-Ricci flow (\ref{2020B}) is equivalent to say for a basic
function $\varphi$%
\[%
\begin{array}
[c]{c}%
\frac{d}{dt}\varphi=\log\det(g_{\alpha\overline{\beta}}^{T}+\varphi
_{\alpha\overline{\beta}})-\log\det(g_{\alpha\overline{\beta}}^{T})-f.
\end{array}
\]

The method to prove this well-posedness of the Sasaki-Ricci flow (\ref{2020B})
is to develop a transverse parabolic theory (\cite{el}). More precisely, let%
\[%
\begin{array}
[c]{c}%
\varphi_{,\alpha\overline{\beta}}=X_{\overline{\beta}}X_{\alpha}%
\varphi-d\varphi(\nabla_{X_{\overline{\beta}}}X_{\alpha})
\end{array}
\]
be the covariant derivative with respect to the metric $g$ and%
\[%
\begin{array}
[c]{c}%
\widetilde{\omega}=d\eta+\frac{1}{2}(\nabla d\varphi(\Phi(\cdot),\cdot)-\nabla
d\varphi(\cdot,\Phi(\cdot))+\frac{1}{2}(\eta\wedge d(\xi\varphi)\circ\Phi
+\eta\wedge d\varphi)).
\end{array}
\]
Then
\[%
\begin{array}
[c]{c}%
\widetilde{\omega}=\sqrt{-1}(g_{i\overline{j}}^{T}+\varphi_{,i\overline{j}%
})dz^{i}\wedge d\overline{z}^{j}.
\end{array}
\]
If $\varphi$ is a basic function, then
\[
\widetilde{\omega}=d\eta+d_{B}d_{B}^{c}\varphi.
\]
Hence one can consider the following equation
\begin{equation}%
\begin{array}
[c]{c}%
\frac{d}{dt}\varphi=\log\det(g_{\alpha\overline{\beta}}^{T}+\varphi
_{,\alpha\overline{\beta}})-\log\det(g_{\alpha\overline{\beta}}^{T})+\xi
^{2}\varphi-f
\end{array}
\label{2025A}%
\end{equation}
for the general function $\varphi$, i.e. $\varphi$ in (\ref{2025A}) need not
to be a basic function. It is easy to check that (\ref{2025A}) is an ordinary
parabolic equation when
\begin{equation}%
\begin{array}
[c]{c}%
\det(g_{\alpha\overline{\beta}}^{T}+\varphi_{,\alpha\overline{\beta}})>0.
\end{array}
\label{2025B}%
\end{equation}
It follows from the parabolic theory that

\begin{lemma}
\label{Al1}For any initial function $\varphi$ with (\ref{2025B}), there exists
a positive $T>0$ and a solution $\varphi(x,t):M\times(0,T)\rightarrow
\mathbb{R}$ of (\ref{2025A}) and (\ref{2025B})) for any $t\in\lbrack0,T)$.
\end{lemma}

Hence in order to prove Proposition \ref{AP1}, the only thing we need to check
is that the equation (\ref{2025A}) preserves the property
\[
\xi\varphi=0.
\]
In fact, since $\xi g_{i\overline{j}}=0,$ we have
\[
\xi(\varphi_{,i\overline{j}})=(\xi\varphi)_{_{,i\overline{j}}}%
\]
and then by straightforward computation
\[%
\begin{array}
[c]{c}%
\frac{d}{dt}(\xi\varphi)^{2}=g(t)^{i\overline{j}}((\xi\varphi)_{,i\overline
{j}}^{2})-2g(t)^{i\overline{j}}((\xi\varphi)_{i}(\xi\varphi)_{\overline{j}%
})+\xi^{2}(\xi\varphi)^{2}-2(\xi^{2}\varphi)^{2}.
\end{array}
\]
It follows from the maximum principle that the flow preserves the property
$\xi\varphi=0$.

Then, by using the standard method as in the paper of Shi \cite{s1}, we have
the short-time solution of the Sasaki-Ricci flow in a complete noncompact
Sasakian $(2n+1)$-manifold with bounded transverse holomorphic bisectional curvature.

\begin{proposition}
\label{AP2}Let $(M,\xi,\eta,g^{T},\Phi,\omega)$ be a complete noncompact
Sasakian $(2n+1)$-manifold with bounded transverse holomorphic bisectional
curvature. For any Sasakian structure $(M,\xi,\eta,g^{T},\Phi,\omega),$ there
exists a family of Sasakian structures $(M,\xi(t),\eta(t),g^{T}(t),\Phi
(t),\omega(t)),t\in\lbrack0,T_{\max})$ for some $T_{\max}>0$ satisfying the
Sasaki-Ricci flow%
\[%
\begin{array}
[c]{lll}%
\frac{d}{dt}g_{i\overline{j}}^{T}(x,t) & = & -R_{i\overline{j}}^{T}(x,t),\\
g_{i\overline{j}}^{T}(x,0) & = & g_{i\overline{j}}^{T}(x)
\end{array}
\]
with the initial condition $(\eta(0),\Phi(0),g(0))=(\eta,\Phi,g).$
\end{proposition}

We recall the Sasaki-Chern-Ricci flow as the following parabolic
Monge-Amp\`{e}re equation (\ref{2025C}) on $[0,T)$
\[%
\begin{array}
[c]{c}%
\left\{
\begin{array}
[c]{ccl}%
\frac{\partial}{\partial t}\varphi & = & \log\frac{(\omega_{0}-tRic^{TC}%
(\omega_{0})+\sqrt{-1}\partial_{B}\overline{\partial}_{B}\varphi)^{n}%
\wedge\eta_{0}}{\omega_{0}^{n}\wedge\eta_{0}},\\
\varphi(0) & = & 0.
\end{array}
\right.
\end{array}
\]
It follows from Proposition \ref{P21}, Lemma \ref{Al1} and the result of
Tosatti-Weinkove \cite{tw}, we have the Sasaki analogue of Chern-Ricci flow.

\begin{theorem}
Let $M$ be a compact Sasakian $(2n+1)$-manifold. Let $S_{A}$ be the supremum
of $S>0$ so that the Sasaki-Chern-Ricci flow has a solution $h^{T}(t)$ with
initial data $h_{0}^{T}$ on $M\times\lbrack0,S)$. Let $S_{B}$ be the supremum
of $S$ for which there is a smooth basic function $u$ such that
\[%
\begin{array}
[c]{c}%
\omega_{0}-tRic^{TC}(\omega_{0})+\sqrt{-1}\partial_{B}\overline{\partial}%
_{B}u>0\}.
\end{array}
\]
Then
\[
S_{A}=S_{B}.
\]

\end{theorem}

Furthermore, it follows from the results of Lott-Zhang \cite{lz} and Lee-Tam
\cite{lt2} that

\begin{theorem}
\label{AT1}Let $M$ be a complete noncompact Sasakian $(2n+1)$-manifold with
bounded curvature. Let $S_{A}$ be the supremum of $S>0$ so that the
Sasaki-Chern-Ricci flow has a solution $h^{T}(t)$ with initial data $h_{0}%
^{T}$ sand $h^{T}(t)$ is uniformly equivalent to $h_{0}^{T}$ on $M\times
\lbrack0,S)$. Let $S_{B}$ be the supremum of $S$ for which there is a smooth
bounded basic function $u$ such that
\[%
\begin{array}
[c]{c}%
\omega_{0}-tRic^{TC}(\omega_{0})+\sqrt{-1}\partial_{B}\overline{\partial}%
_{B}u\geq\beta\omega_{0}%
\end{array}
\]
for some $\beta>0$. Assume that

\begin{enumerate}
\item $|T|_{h_{0}^{T}}$ and $|\overline{\partial}T|_{h_{0}^{T}}$ are uniformly bounded;

\item the transverse bisectional curvature of $h_{0}^{T}$ is uniformly bounded
below (we do not assume that the bounded curvature of $h_{0}^{T}$);

\item there exists a smooth real basic function $\rho$ which is uniformly
equivalent to the distance function from a fixed point such that
$|\partial_{B}\rho|_{h_{0}^{T}}$ and $|\partial_{B}\overline{\partial}_{B}%
\rho|_{h_{0}^{T}}$ are uniformly bounded. Then%
\[
S_{A}=S_{B}.
\]

\end{enumerate}
\end{theorem}

\begin{remark}
It is in general not true that $h_{0}^{T}$ has bounded geometry of infinite order.
\end{remark}

Then we need the following key Lemma:

\begin{lemma}
There exists $0<\alpha<1$ depending only on $n$ so that the following is true:
Let $(M,\xi,\eta,g_{0}^{T},\Phi,\omega_{0})$ be a noncompact Sasakian
$(2n+1)$-manifold with the precompact open set $U\subset M$. For $\rho>0$ so
that $U_{\rho}$ is the nonempty set such that $B_{g_{0}^{T}}(x,\rho
)\subset\subset U$ with \
\[%
\begin{array}
[c]{c}%
|Rm^{T}|(x)\leq\rho^{-2}%
\end{array}
\]
and
\[%
\begin{array}
[c]{c}%
\mathrm{inj}_{g_{0}^{T}}(x)\geq\rho
\end{array}
\]
for all $x\in U$. Then for any component $X$ of $U_{\rho}$, there is a
solution $g^{T}(t)$ to the Sasaki-Ricci flow on $X\times\lbrack0,\alpha
\rho^{2}\}$ such that $g^{T}(t)$ satisfies

\begin{enumerate}
\item $g^{T}(0)=g_{0}^{T}$ on $X,$

\item
\[
\alpha g_{0}^{T}\leq g^{T}(t)\leq\alpha^{-1}g_{0}^{T}%
\]
on $X\times\lbrack0,\alpha\rho^{2}\}.$
\end{enumerate}
\end{lemma}

The key fact is: Assume that $\rho=1.$ There is a smooth \ basic function
$\sigma(x)\geq0$ on $U$ such that $\sigma(x)=0$ on $U_{1}$ and $\sigma
(x)\geq1$ on $\partial U$ and $|\nabla\sigma(x)|+|\nabla^{2}\sigma(x)|\leq C.$
Let $W=\{x\in U|\sigma(x)<1\},X\subset U_{1}\subset W.$ Let $\overline{W}$ be
the component of $W$ containing $X.$ Let
\[
h_{0}=e^{2F}g_{0}^{T}%
\]
be the transverse Hermitian metric (\cite[Lemma B.1.]{ct2}) on $\overline{W}$
for some $F(x)=\mathcal{F(}\sigma(x))$ which is defined by \cite[Lemma
4.1.]{lt2}. Then $(\overline{W},h_{0})$ is a complete transverse Hermitian
metric and $(U_{\rho},h_{0})$ has bounded geometry of infinite order. It
follows from Theorem \ref{AT1} that the Sasaki-Chern-Ricci flow has a solution
$h^{T}(t)$ on $\overline{W}\times\lbrack0,\alpha]$ for some $\alpha>0$ such
that
\[
\alpha h_{0}^{T}\leq h(t)\leq\alpha^{-1}h_{0}^{T}.
\]

Furthermore, it follows from Lee-Tam \cite{lt1} and Simon-Topping \cite{st} that

\begin{lemma}
For any $v_{0},$ there exist $\overline{S}(n,v_{0})>0,$ $C_{0}(n,v_{0})>0$
such that the following holds: Suppose that $(M,\xi(t),\eta(t),g^{T}%
(t),\Phi(t))$ is a Sasaki-Ricci flow for $t\in\lbrack0,T]$ and $x_{0}\in M$
such that $B_{g_{t}^{T}}(x_{0},r)\subset\subset M$. If \
\[%
\begin{array}
[c]{c}%
BK(g^{T}(t))\geq-r^{-2}\text{ \textrm{and}\ }V_{g_{0}^{T}}(x_{0},r)\geq
v_{0}r^{2n}%
\end{array}
\]
on $B_{g_{t}^{T}}(x_{0},r),t\in\lbrack0,S].$ Then%
\[%
\begin{array}
[c]{c}%
|Rm^{T}(x,t)|\leq\frac{C_{0}}{t}\text{\ \textrm{and}\ }|\nabla Rm^{T}%
(x,t)|\leq\frac{C_{0}}{t^{\frac{3}{2}}}%
\end{array}
\]
on $B_{g_{t}^{T}}(x_{0},\frac{1}{8}r)$ and $t\in(0,S]\cap(0,\overline{S}%
r^{2}).$ Moreover,
\[%
\begin{array}
[c]{c}%
\mathrm{inj}_{_{g_{t}^{T}}}(x)\geq(C^{-1}t)^{\frac{1}{2}}%
\end{array}
\]
on $B_{g_{t}^{T}}(x_{0},\frac{1}{8}r),t\in(0,S]\cap(0,\overline{S}r^{2}).$
\end{lemma}

Finally, as in the paper of Lee-Tam \cite[Theorem 5.1.]{lt2}, we have the
Sasaki analogue of the short-time solution for Sasaki-Ricci flow in a complete
noncompact Sasakian $(2n+1)$-manifold without the boundedness of transverse
nonnegative holomorphic bisectional curvature.

\begin{proposition}
Let $(M,\xi,\eta,g_{0}^{T},\Phi,\omega_{0})$ be a complete noncompact Sasakian
$(2n+1)$-manifold with transverse nonnegative holomorphic bisectional
curvature and
\[%
\begin{array}
[c]{c}%
V_{g_{0}^{T}}(x,1)\geq v_{0}%
\end{array}
\]
for some $v_{0}>0$ and all $x\in M$. Then there exist $T_{0}(n,v_{0})>0,$
$C_{0}(n,v_{0})>0$ such that a family of Sasakian structures $(M,\xi
(t),\eta(t),g^{T}(t),\Phi(t),\omega(t))$ satisfying the Sasaki-Ricci flow
\[%
\begin{array}
[c]{lll}%
\frac{d}{dt}g_{i\overline{j}}^{T}(x,t) & = & -R_{i\overline{j}}^{T}(x,t),\\
g_{i\overline{j}}^{T}(x,0) & = & g_{i\overline{j}}^{T}(x).
\end{array}
\]
with
\[%
\begin{array}
[c]{c}%
||Rm^{T}||^{2}(x,t)\leq\frac{C_{0}}{t},\text{\ \textrm{on} }M\times
\lbrack0,T_{0}).
\end{array}
\]

\end{proposition}

\end{document}